\documentclass[a4paper,12pt]{article}

\usepackage{amsmath,amstext,amscd,amsthm,amsfonts,diagrams}
\usepackage{amssymb}
\newtheorem{theorem}{Theorem}[section]
\newtheorem{lemma}[theorem]{Lemma}
\newtheorem{proposition}[theorem]{Proposition}
\newtheorem{definition}[theorem]{Definition}
\newtheorem{corollary}[theorem]{Corollary}
\theoremstyle{remark}
\newtheorem{rk}[theorem]{Remark}

\newcommand{\quot}{\ensuremath{/ \hspace{-1.2mm}/}}

\def\Ad{\mathop{\rm Ad}\nolimits}
\def\ad{\mathop{\rm ad}\nolimits}
\def\Int{\mathop{\rm Int}\nolimits}
\def\Aut{\mathop{\rm Aut}\nolimits}

\def\Lie{\mathop{\rm Lie}\nolimits}
\def\GL{\mathop{\rm GL}\nolimits}
\def\SL{\mathop{\rm SL}\nolimits}

\def\SO{\mathop{\rm SO}\nolimits}

\def\Sp{\mathop{\rm Sp}\nolimits}

\def\charac{\mathop{\rm char}\nolimits}
\def\dim{\mathop{\rm dim}\nolimits}

\def\Hom{\mathop{\rm Hom}\nolimits}

\def\im{\mathop{\rm im}\nolimits}
\def\ker{\mathop{\rm ker}\nolimits}
\def\End{\mathop{\rm End}\nolimits}

\def\rank{\mathop{\rm rk}\nolimits}
\def\min{\mathop{\rm min}\nolimits}

\def\rank{\mathop{\rm rk}\nolimits}
\def\Spec{\mathop{\rm Spec}\nolimits}
\def\Spin{\mathop{\rm Spin}\nolimits}

\def\diag{\mathop{\rm diag}\nolimits}

\def\Frac{\mathop{\rm Frac}\nolimits}
\def\dom{\mathop{\rm dom}\nolimits}

\title{Vinberg's $\theta$-groups in positive characteristic and Kostant-Weierstrass slices}
\author{Paul Levy \\
paul.levy@epfl.ch}

\begin{document}

\bibliographystyle{plain}

\maketitle{}{}

\begin{abstract}
We generalize the basic results of Vinberg's $\theta$-groups, or periodically graded reductive Lie algebras, to fields of good positive characteristic.
To this end we clarify the relationship between the little Weyl group and the (standard) Weyl group.
We deduce that the ring of invariants associated to the grading is a polynomial ring.
This approach allows us to prove the existence of a KW-section for a classical graded Lie algebra (in zero or good characteristic), confirming a conjecture of Popov in this case.
\end{abstract}

\section{Introduction}

Classical results of invariant theory relate the geometry of the adjoint representation of a reductive group to familiar properties of elements of the Lie algebra.
In particular, Cartan subalgebras, Weyl groups and semisimple and nilpotent elements appear naturally in the description of invariants, closed orbits and fibres of the quotient map.
On the other hand, there are many circumstances in which the concepts of Cartan subalgebra, Weyl group and nilpotent cone have analogues with similar properties.
In \cite{vin}, Vinberg studied such generalizations for representations arising from periodic gradings of complex reductive Lie algebras.
Specifically, let $G$ be a complex reductive group, let ${\mathfrak g}=\Lie(G)$, let $\theta$ be an automorphism of $G$ of order $m$ and let $\zeta=e^{2\pi i/m}$.
There is a grading of ${\mathfrak g}$ induced by $d\theta$:
$${\mathfrak g}=\bigoplus_{i\in{\mathbb Z}/m{\mathbb Z}}{\mathfrak g}(i),\;\;\mbox{where }{\mathfrak g}(i)=\{ x\in{\mathfrak g}\,|\,d\theta(x)=\zeta^ix\}$$
Clearly $[{\mathfrak g}(i),{\mathfrak g}(j)]\subset{\mathfrak g}(i+j)$ for any $i,j\in{\mathbb Z}/m{\mathbb Z}$.
Let $G(0)=(G^\theta)^\circ$.
Then $\Lie(G(0))={\mathfrak g}(0)$ and $G(0)$ normalizes ${\mathfrak g}(1)$.
A {\it Cartan subspace} of ${\mathfrak g}(1)$ is a maximal commutative subspace consisting of semisimple elements.
The principal results of \cite{vin} are:

 - any two Cartan subspaces of ${\mathfrak g}(1)$ are $G(0)$-conjugate and any semisimple element is contained in a Cartan subspace,
 
 - the $G(0)$-orbit through $x\in{\mathfrak g}(1)$ is closed if and only if $x$ is semisimple,

 - the embedding ${\mathfrak c}\hookrightarrow{\mathfrak g}(1)$ induces an isomorphism $k[{\mathfrak g}(1)]^{G(0)}\rightarrow k[{\mathfrak c}]^{W_{\mathfrak c}}$, where ${\mathfrak c}$ is any Cartan subspace of ${\mathfrak g}(1)$ and $W_{\mathfrak c}=N_{G(0)}({\mathfrak c})/Z_{G(0)}({\mathfrak c})$,
 
 - the little Weyl group $W_{\mathfrak c}$ is generated by pseudoreflections and hence $k[{\mathfrak c}]^{W_{\mathfrak c}}$ is a polynomial ring. 

In the case of an involution, the grading ${\mathfrak g}={\mathfrak g}(0)\oplus{\mathfrak g}(1)$ is known as the symmetric space decomposition and has been studied extensively, especially since the seminal work of Kostant and Rallis \cite{kostrall}.
In this case the little Weyl group is itself a Weyl group (for a root system which can be related in a natural way to the root system of $G$).
Hence, while the geometric properties of symmetric spaces are quite close to those of the adjoint representation, the $m>2$ case is more interesting from the point of view of reflection groups.
Most of the results of Kostant-Rallis are now known to hold in good positive characteristic by work of the author \cite{invs}.
However, with the exception of \cite{kawanaka} (and perhaps \cite{panorbits}), there has been little work on (general) $\theta$-groups in positive characteristic.
The first main task of this paper will be to extend Vinberg's above-mentioned results to the case where $G$ is a reductive group over a field of good positive characteristic $p$ not dividing $m$.
The major obstacles concern separability of the quotient morphism ${\mathfrak g}(1)\rightarrow{\mathfrak g}(1)\quot G(0):=\Spec (k[{\mathfrak g}(1)]^{G(0)})$ and the failure of the Shephard-Todd theorem in positive characteristic.
The former problem can be resolved by a careful analysis of the centralizer of a Cartan subspace.
To show that the little Weyl group is generated by pseudoreflections and that its ring of invariants is polynomial, we prove directly that Vinberg's description \cite[\S 7]{vin} of $W_{\mathfrak c}$ for $G$ of classical type holds in good positive characteristic, and apply a result of Panyushev and an inspection of orders of centralizers in Weyl groups for the exceptional types.
While our approach requires somewhat more work than that of \cite{vin}, it makes the relationship between the little Weyl group and the Weyl group of $G$ clear.
This allows us to prove for classical graded Lie algebras a long-standing conjecture in this field, the existence of a slice in ${\mathfrak g}(1)$ analogous to Kostant's slice to the regular orbits in ${\mathfrak g}$.

We provide the following criterion for a Cartan subspace to be contained in the centre of ${\mathfrak g}$.
An automorphism $\theta$ is of zero rank if any element of ${\mathfrak g}(1)$ is nilpotent.

\begin{lemma}\label{zsintro}
Suppose $p>2$.
Then the following are equivalent: (i) ${\mathfrak g}(1)$ contains no non-central semsimple elements, (ii) $\theta|_{G^{(1)}}$ is either of order less than $m$ or is of zero rank, (iii) ${\mathfrak g}(1)={\mathfrak s}\oplus{\mathfrak n}$, where ${\mathfrak s}$ (resp. ${\mathfrak n}$) is the set of semisimple (resp. nilpotent) elements of ${\mathfrak g}(1)$ and ${\mathfrak s}\subseteq{\mathfrak z}({\mathfrak g})$.
\end{lemma}

We remark that the above result fails if $p=2$.
We prove the following Lemma by some simple geometric arguments.

\begin{lemma}\label{csintro}
Let ${\mathfrak c}$ be a Cartan subspace of ${\mathfrak g}(1)$.

(i) The morphism $G(0)\times {\mathfrak z}_{\mathfrak g}({\mathfrak c})\rightarrow{\mathfrak g}(1)$ is dominant and separable.

(ii) Any two Cartan subspaces of ${\mathfrak g}(1)$ are conjugate by an element of $G(0)$.
\end{lemma}

If $\theta$ is an involution and $T\subset G$ is a $\theta$-stable torus, then it is not difficult to see that $T=T_+\cdot T_-$, where $T_+=(T^\theta)^\circ$ and $T_-=\{ t\in T|\theta(t)=t^{-1}\}^\circ$.
Moreover, the Lie algebras of $T_+$ and $T_-$ are, respectively, the $(+1)$ and $(-1)$ eigenspace for the differential of $\theta$ on $\Lie(T)$ \cite[p.290]{richinvs}.
An important tool in our analysis will be a generalization of this decomposition to arbitrary $m$.
Roughly speaking, one decomposes $T$ as a product of subtori $T_d$, $d|m$, such that `the minimal polynomial of $e^{2\pi di/m}$ applied to $\theta$' acts trivially on $T_d$.

\begin{lemma}
Let $T$ be a $\theta$-stable torus and let ${\mathfrak t}=\Lie(T)$.
We have $T=\prod_{d|m} T_d$ (see Lemma \ref{stabletori} for definitions) and ${\mathfrak t}=\oplus_{d|m}\Lie(T_d)$.
Moreover, $\Lie(T_d)=\oplus_{(i,m)=d}{\mathfrak t}\cap{\mathfrak g}(i)$.
In particular $T_m=(T^\theta)^\circ$.
\end{lemma}

We turn next to consideration of the quotient morphism $\pi:{\mathfrak g}(1)\rightarrow{\mathfrak g}(1)\quot G(0)$.
Recall that each fibre of $\pi$ contains a unique closed orbit (which is also the unique orbit of minimal dimension) and for $x\in{\mathfrak g}(1)$, $\pi(x)=\pi(0)$ if and only if $0$ is contained in the closure of $G(0)\cdot x$.
Arguing in a similar manner to \cite{vin} we obtain:

\begin{lemma}
If $x\in {\mathfrak g}(1)$ then $G(0)\cdot x$ is closed if and only if $x$ is semisimple.
On the other hand, $0$ is contained in the closure of $G(0)\cdot x$ if and only if $x$ is nilpotent.
\end{lemma}

In general the quotient morphism for the action of a reductive group on an affine variety need not be separable.
Here we face a certain difficulty because a separability criterion established by  Richardson \cite[9.3]{richinvs} does not in general hold.
However, Lemmas \ref{zsintro} and \ref{csintro} allow us (after a little work) to adapt Richardson's arguments to the present circumstances.

\begin{lemma}
Assume $p>2$.
Then $k({\mathfrak g}(1))^{G(0)}$ is the fraction field of $k[{\mathfrak g}(1)]^{G(0)}$ and hence $\pi:{\mathfrak g}(1)\rightarrow{\mathfrak g}(1)\quot G(0)$ is a separable morphism.
\end{lemma}

We can then employ some fairly standard invariant theoretic arguments to generalize Vinberg's version \cite[Thm. 7]{vin} of the Chevalley Restriction Theorem.
Let ${\mathfrak c}$ be a Cartan subspace of ${\mathfrak g}(1)$.
We denote by $W_{\mathfrak c}$ the little Weyl group $N_{G(0)}({\mathfrak c})/Z_{G(0)}({\mathfrak c})$.

\begin{theorem}
Suppose $p>2$.
Then the embedding ${\mathfrak c}\hookrightarrow{\mathfrak g}(1)$ induces an isomorphism ${\mathfrak c}/W_{\mathfrak c}\rightarrow{\mathfrak g}(1)\quot G(0)$.
\end{theorem}

Next we turn to the consideration of the group $W_{\mathfrak c}$.
For $G$ of classical type Vinberg gave a precise description of the little Weyl group.
The basic approach of \cite[\S 7]{vin} is to classify inner automorphisms $\Int g$ of $G$ by considering the eigenvalues of $g$, and similarly for outer automorphisms.
In essence, this perspective fixes a maximal torus of ($G$ containing a maximal torus of) $G(0)$.
Here we follow a different approach more in common with the classification of involutions (see \cite{springerinvs} or \cite{helminck}): we fix a (suitable) $\theta$-stable maximal torus $T$ whose Lie algebra contains a Cartan subspace.
Hence we describe an inner automorphism as $\Int n_w$, where $n_w\in N_G(T)$ and $w=n_wT\in W$ is an element of order $m$ (and similarly for outer automorphisms).
This allows us to relate $W_{\mathfrak c}$ to the centralizer of $w$ in $W$.

 - If $G$ is of classical type then $W_{\mathfrak c}$ is of the form $G(m',1,r)$ or $G(m',2,r)$ where $m'\in\{ m/2,m,2m\}$ (cf. \cite{vin}).
 
 - If $G$ is of exceptional type and $m>2$ or if $G$ is of type $D_4$ and $\charac k=p>3$ then the order of $W_{\mathfrak c}$ is coprime to $p$.

This, along with a reduction theorem to the almost simple case (\S 3) and application of a result of Panyushev \cite{panorbits} gives us the following result for any $G$ satisfying the `standard hypotheses' (see \S 3).

\begin{theorem}
The group $W_{\mathfrak c}$ is generated by pseudoreflections and ${\mathfrak c}/W_{\mathfrak c}$ is isomorphic to a vector space of dimension $r=\dim{\mathfrak c}$.
\end{theorem}

Recall that a Kostant-Weierstrass slice or KW-section for $(G,\theta)$ is a linear subvariety ${\mathfrak v}$ of ${\mathfrak g}(1)$ for which the restriction of functions $k[{\mathfrak g}(1)]^{G(0)}\rightarrow k[{\mathfrak v}]$ is an isomorphism.
The existence of KW-sections for $\theta$-groups is a long-standing conjecture of Popov in characteristic zero \cite{popov}.
In \cite[Cor. 5]{pansemislice} Panyushev proved that a KW-section exists if ${\mathfrak g}(0)$ is semisimple.
More recently, Panyushev proved in \cite[Thm. 3.5]{panslice} that KW-sections exist for `N-regular' gradings, that is, those such that ${\mathfrak g}(1)$ contains a regular nilpotent element of ${\mathfrak g}$.
Here we prove existence of a KW-section for a classical graded Lie algebra in zero or good positive characteristic.
Our approach to describing the little Weyl group makes it clear that if $G$ is of classical type then there is an N-regular minimal $\theta$-stable semisimple subgroup $L$ of $G$ whose Lie algebra contains ${\mathfrak c}$ and such that all elements of $W_{\mathfrak c}$ have representatives in $L(0)$.
The proof of Popov's conjecture for classical graded Lie algebras can therefore be reduced to the subgroups $L$ constructed in this way.
The solution in characteristic zero is then immediate due to Panyushev \cite{panslice}; in positive characteristic we generalize Panyushev's result by a similar reasoning.

\begin{theorem}
Let $\charac k=0$ or $p>2$ and let $G$ be of classical type, that is, one of $\GL(n,k)$, $\SL(n,k)$, $\SO(n,k)$, $\Sp(2n,k)$.
Then the grading of ${\mathfrak g}$ induced by $\theta$ admits a KW-section.
\end{theorem}

{\it Notation.}
For $G$ an affine algebraic group, we denote by $\Int g$ the corresponding inner automorphism of $G$, by $G^\circ$ the connected component of $G$ and by $G^{(1)}$ the derived subgroup of $G$.
If $\theta$ an automorphism of $G$ then denote by $G^\theta$ the isotropy subgroup of $G$.
Write $x=x_sx_u$ (resp. $x=x_s+x_n$) for the Jordan-Chevalley decomposition of $x\in G$ (resp. $x\in\Lie(G)$).
We denote by $[n/m]$ the integer part of the fraction $n/m$.

{\it Acknowledgement.}
I would like to thank Alexander Premet for many helpful remarks and conversations, and Ross Lawther for crucial advice concerning conjugacy classes in exceptional type Weyl groups.
I would also like to express appreciation for the helpful comments of Dmitri Panyushev.

\section{Preliminaries}

Let $\Phi$ be an irreducible root system with basis $\Delta=\{ \alpha_1,\ldots ,\alpha_r\}$.
Recall that $p$ is {\it good} for $\Phi$ if for any $\alpha=\sum_{i=1}^r m_i\alpha_i\in\Phi$, $p>|m_i|$ for all $i$.
Specifically, 2 is a bad prime for all irreducible root systems other than type $A$, 3 is bad for all exceptional type root systems and 5 is bad for type $E_8$; otherwise $p$ is good.
More generally, $p$ is good for a root system $\Phi$ if it is good for each irreducible component of $\Phi$, and is good for a reductive algebraic group if it is good for its root system.

Let $G$ be a reductive affine algebraic group over the algebraically closed field $k$ of characteristic $p>0$ and let ${\mathfrak g}=\Lie(G)$.
We assume throughout that $p$ is good for $G$.
It is well-known that $p$ is good for any Levi subgroup of $G$.
In fact it is straightforward to see that $p$ is good for any {\it pseudo}-Levi subgroup of $G$.
(A pseudo-Levi subgroup of $G$ is a subgroup of the form $Z_G(s)^\circ$, where $s\in G$ is semisimple.
The possible root systems for such subgroups are given by proper subsets of the extended Dynkin diagram of $G$, see \cite[Prop. 2]{sommers} in characteristic zero, \cite[Prop. 20]{mcninchsommers} in positive characteristic.)

Recall that the Lie algebra of any affine algebraic group over $k$ is {\it restricted}.
Hence there is a map $[p]:{\mathfrak g}\rightarrow{\mathfrak g}$, $x\mapsto x^{[p]}$ such that:

 - $\ad x^{[p]}=(\ad x)^p$ for all $x\in{\mathfrak g}$,
 
 - the map $\xi:{\mathfrak g}\rightarrow U({\mathfrak g})$, $x\mapsto x^p-x^{[p]}$ is semilinear, that is $\xi(\lambda x+y)=\lambda^p\xi(x)+\xi(y)$ for all $x,y\in{\mathfrak g}$, $\lambda\in k$.
 
We denote by $x\mapsto x^{[p^i]}$ the $i$-th iteration of ${[p]}$.
Recall also that $x\in{\mathfrak g}$ is semisimple if and only if $x\in\sum_{i\geq 1}kx^{[p^i]}$, and is nilpotent if and only if $x^{[p^N]}=0$ for large enough $N$.

Let $\theta:G\rightarrow G$ be an automorphism of order $m$, $p\nmid m$ and let $d\theta:{\mathfrak g}\rightarrow{\mathfrak g}$ be the corresponding restricted Lie algebra automorphism of ${\mathfrak g}$.
Fix once and for all a primitive $m$-th root of unity $\zeta$ in $k$.
Then there is a direct sum decomposition ${\mathfrak g}={\mathfrak g}(0)\oplus{\mathfrak g}(1)\oplus\ldots\oplus{\mathfrak g}(m-1)$, where ${\mathfrak g}(i)=\{ x\in{\mathfrak g}\, |\, d\theta(x)=\zeta^ix\}$.
In fact, this is a ${\mathbb Z}/m{\mathbb Z}$-grading of ${\mathfrak g}$: if $x\in{\mathfrak g}(i)$, $y\in{\mathfrak g}(j)$ then $[x,y]\in {\mathfrak g}(i+j)$ ($i,j\in{\mathbb Z}/m{\mathbb Z}$).
Let $G(0)=(G^\theta)^\circ$.
Then $G(0)$ is reductive \cite[8.1]{steinberg} and $\Lie(G(0))={\mathfrak g}(0)$ \cite[9.1]{borel}.
Clearly the adjoint action of $G(0)$ stabilises each of the subspaces ${\mathfrak g}(i)$.

We are interested in the properties of the $G(0)$-representation ${\mathfrak g}(1)$.
Note that the action of $G(0)$ on any ${\mathfrak g}(i)$ ($i\neq 0$) can be reduced to this case.
Indeed, if $0<i<m$ then let $\psi=\theta^{(m,i)}$, let $\overline{G}=(G^\psi)^\circ$ and let $\overline{\mathfrak g}:=\Lie(\overline{G})$ (cf. \cite[\S 2.1]{vin}).
Then $\overline{G}$ is $\theta$-stable, reductive and contains $G(0)$, and $\overline{\mathfrak g}=\sum_{0\leq j<m/(m,i)}\overline{\mathfrak g}(j)$, where $\overline{\mathfrak g}(j)={\mathfrak g}(ij)$.
In particular, $\overline{\mathfrak g}(1)={\mathfrak g}(i)$.
It can be easily checked that the condition $p\nmid m$ implies that $p$ is good for $\overline{G}$ and $G(0)$.

\begin{lemma}\label{first}
(a) Let $x\in{\mathfrak g}$, and let $x=x_s+x_n$ be the Jordan-Chevalley decomposition of $x$.
Then $x\in{\mathfrak g}(i)$ if and only if $x_s,x_n\in{\mathfrak g}(i)$.

(b) If $x\in{\mathfrak g}(i)$ then $x^{[p]}\in{\mathfrak g}(ip)$.
\end{lemma}

\begin{proof}
For any (rational) automorphism $\theta$ of $G$, $d\theta(x^{[p]})=d\theta(x)^{[p]}$, hence (b) is immediate.
Since any restricted Lie algebra automorphism of ${\mathfrak g}$ preserves semisimplicity and nilpotency, $d\theta(x_s)$ (resp. $d\theta(x_n)$) is semisimple (resp. nilpotent) and $[d\theta(x_s),d\theta(x_n)]=0$.
Hence $d\theta(x)=d\theta(x_s)+d\theta(x_n)$ is the Jordan-Chevalley decomposition of $d\theta(x)$.
This proves (a).
\end{proof}

The following result of Steinberg \cite[7.5]{steinberg} is essential to any discussion of automorphisms of $G$.
(This was earlier proved for connected $H$ by Winter \cite{winter}.)

{\it - For any rational automorphism $\sigma$ of a linear algebraic group $H$ there exists a $\sigma$-stable Borel subgroup of $H$.
If $\sigma$ is semisimple then there is a $\sigma$-stable maximal torus of $H$ contained in a $\sigma$-stable Borel subgroup.}

Following Springer for the case $m=2$, we call a pair $(B,T)$, $B$ a $\theta$-stable Borel subgroup of $G$ and $T$ a $\theta$-stable maximal torus of $B$ a {\it fundamental pair}.
Let $\Phi=\Phi(G,T)$ be the roots of $G$ relative to $T$, let $\Phi^+$ be the positive system in $\Phi$ associated to $B$ and let $\Delta$ be the corresponding basis for $\Phi$.
For each $\alpha\in\Phi$, denote by $\alpha^\vee$ the corresponding coroot.
Let $X(T):=\Hom (T,k^\times)$ and let $Y(T):=\Hom(k^\times,T)$.
Consider the coroots as elements of $Y(T)$ via the perfect pairing $\langle .\, ,.\rangle:X(T)\times Y(T)\rightarrow {\mathbb Z}$.
Let $\gamma$ be the graph automorphism of $\Phi$ induced by $\theta$ (that is, such that $d\theta({\mathfrak g}_\alpha)={\mathfrak g}_{\gamma(\alpha)}$).
Then $\gamma$ permutes the elements of $\Delta$.
Let $\{ h_\alpha,e_\beta:\alpha\in\Delta,\beta\in\Phi\}$ be a Chevalley basis for $[{\mathfrak g},{\mathfrak g}]$.
(In fact the $h_\alpha=[e_\alpha,e_{-\alpha}]=d\alpha^\vee(1)$ need not be linearly independent (or even non-zero!), but this problem can be solved by removing some of the $h_\alpha$.
We remark that the error of assuming that the $h_\alpha$ are linearly independent and span $\Lie(T\cap G^{(1)})$ appears in the work of the author on involutions \cite[p. 512]{invs}.
This error can easily be remedied by applying Lemma \ref{redcase}(b) below to pass from $[{\mathfrak g},{\mathfrak g}]$ to all of ${\mathfrak g}$.)

There exist constants $c(\alpha)\in k^\times$, $\alpha\in\Phi$, such that:

 - $d\theta(e_\alpha)=c(\alpha)e_{\gamma(\alpha)}$, $\alpha\in\Phi$,
 
 - $d\theta(h_\alpha)=h_{\gamma(\alpha)}$, $\alpha\in\Delta$,
 
 - $c(\alpha)c(-\alpha)=1$, $\alpha\in\Phi$,
 
 - $c(\alpha)c(\gamma(\alpha))\ldots c(\gamma^{m-1}(\alpha))=1$.
 
The second statement follows immediately from the fact that $h_\alpha=d\alpha^\vee(1)$.
But $h_\alpha=[e_\alpha,e_{-\alpha}]$, hence the third statement also follows.

Note that after conjugation by $\Ad t$ for some $t\in T$, we may assume that $c(\alpha)=1$ for any $\alpha\in\Delta$ such that $\gamma(\alpha)\neq \alpha$.
Following Kawanaka \cite{kawanaka} let $l(\alpha)$ denote the cardinality of the set $(\alpha)=\{ \alpha,\gamma(\alpha),\ldots ,\gamma^{m-1}(\alpha)\}$ and let $C(\alpha)=c(\alpha)c(\gamma(\alpha))\ldots c(\gamma^{l(\alpha)-1}(\alpha))$.
Then $C(\alpha)^{m/l(\alpha)}=1$.
Let $n(\alpha)$ denote the order of $C(\alpha)$ (as a root of unity) and let ${\mathfrak g}_{(\alpha)}=\sum_{\beta\in(\alpha)}{\mathfrak g}_\alpha$.
It is easy to verify that:

 - $\dim {\mathfrak g}_{(\alpha)}\cap{\mathfrak g}(1)= \left\{ \begin{array}{ll} 1 & \mbox{if $n(\alpha)=m/l(\alpha)$}, \\ 0 & \mbox{otherwise.}\end{array}\right.$

The following lemma first appeared in \cite[2.2.5]{kawanaka}, with the slight error that case (ii) for $l(\alpha)>2$ was omitted.
The results of \cite{kawanaka} remain valid simply by modifying the definition \cite[p. 582]{kawanaka} of $w_{(\alpha)}$: in case (ii) $w_{(\alpha)}=\prod_{i=1}^{l(\alpha)/2-1} w_{\gamma^i(\alpha)} w_{\gamma^{i+l(\alpha)/2}(\alpha)}w_{\gamma^i(\alpha)}$.
This is consistent with \cite{kawanaka} in the case $l(\alpha)=2$.
 
\begin{lemma}\label{kawanaka}
For $\alpha\in\Phi$, one of the following two cases occurs:

(i) $l(\alpha)=1$, or $l(\alpha)\geq 2$ and any two roots in $(\alpha)$ are orthogonal,
 
(ii) $l(\alpha)$ is even, the elements of $(\alpha)$ generate a subsystem of $\Phi$ of type $A_2^{l(\alpha)/2}$, and $\langle\alpha,\gamma^{l(\alpha)/2}(\alpha)\rangle=-1$.
 \end{lemma}

\begin{proof}
This follows from the classification of root systems and the fact that $\gamma$ induces an automorphism of the subsystem of $\Phi$ spanned by the roots in $(\alpha)$.
\end{proof}

We deduce that:

\begin{lemma}\label{reg0}
Let $S$ be a maximal torus of $G(0)$.
Then $S$ is regular in $G$.
\end{lemma}

\begin{proof}
With the above description of $\theta$, $\alpha^\vee(t)\gamma(\alpha)^\vee(t)\dots\gamma^{l(\alpha)-1}(\alpha)^\vee(t)\in G(0)$ for all $\alpha\in\Phi^+$, $t\in k^\times$.
But then we may assume that $S$ contains the torus generated by all $(\alpha^\vee+\gamma(\alpha)^\vee+\ldots+\gamma^{l(\alpha)-1}(\alpha)^\vee)(k^\times)$, $\alpha\in\Phi$.
Since $\langle \alpha^\vee+\gamma(\alpha)^\vee+\ldots+\gamma^{l(\alpha)-1}(\alpha)^\vee,\alpha\rangle\neq 0$ by Lemma \ref{kawanaka}, ${\mathfrak g}^S={\mathfrak h}$ and hence $S$ is regular in $G$.
\end{proof}

We make the following slight modification to \cite[Lemma 1.1]{invs}.
The only difference is the final statement of (a) (which is immediate since $\mu_m$ is a group of order prime to $p$) and the inclusion of (b), which is proved in exactly the same way as (a).

\begin{lemma} \label{redcase}
(a) Let $\theta$ be an automorphism of $G$ of order $m$, $p\nmid m$, let $T$ be a $\theta$-stable maximal torus of $G$ and let ${\mathfrak t}=\Lie(T)$, ${\mathfrak t}'=\Lie(T\cap G^{(1)})$.
There exists a $\theta$-stable toral algebra ${\mathfrak s}$ such that ${\mathfrak t}={\mathfrak t}'\oplus{\mathfrak s}$, and hence ${\mathfrak g}={\mathfrak g}'\oplus{\mathfrak s}$ (vector space direct sum).

If $m|(p-1)$, then we can choose a toral basis for ${\mathfrak s}$ consisting of eigenvectors for $d\theta$.
More generally, ${\mathfrak s}^{tor}$ decomposes as a sum of irreducible ${\mathbb F}_p[\mu_m]$-modules (where $\mu_m$ here denotes the cyclic group of order $m$).

(b) The above statements all hold if one replaces ${\mathfrak t}'$ by ${\mathfrak t}''={\mathfrak t}\cap[{\mathfrak g},{\mathfrak g}]$.
\end{lemma}

\begin{rk}
It is perhaps instructive to give an explicit description of the irreducible ${\mathbb F}_p[C_m]$-modules: let $\sigma$ be a generator for $\mu_m$.
Then $V$ is irreducible if and only if it has a basis $v_1,v_2,\ldots ,v_r$ ($r\mid m$) such that $\sigma (v_i)=v_{i+1}$ ($1\leq i<r$) and $\sigma(v_r)= lv_1$ for some $l\in{\mathbb F}^\times_p$ of order $m/r$.
\end{rk}

We will also need the following result of Steinberg.

\begin{lemma}\label{cover}
Suppose $G$ is semisimple and $\pi:\hat{G}\rightarrow G$ is the universal covering of $G$.
There exists a unique automorphism $\hat\theta$ of $\hat{G}$ such that the following diagram commutes:
\begin{diagram} \hat{G} & \rTo^{\hat\theta} & \hat{G} \\ \dTo^\pi & & \dTo^\pi \\
G & \rTo^\theta & G \end{diagram}
Moreover, $\hat\theta$ is of order $m$.
\end{lemma}

\begin{proof}
Existence and uniqueness are proved in \cite[9.16]{steinberg}.
It follows immediately that $\hat\theta$ has the same order as $\theta$.
\end{proof}

\begin{lemma}\label{triv}
Suppose the order of $\theta|_{G^{(1)}}$ is strictly less than $m$.
Then ${\mathfrak g}(1)\subset{\mathfrak z}({\mathfrak g})$.
\end{lemma}

\begin{proof}
Recall that any nilpotent element of ${\mathfrak g}$ is contained in ${\mathfrak g}'=\Lie(G^{(1)})$.
(This follows from, for example \cite[14.26 \& 11.3(2)]{borel}.)
But therefore if $\theta|_{G^{(1)}}$ is of order $m'<m$ then there are no nilpotent elements in ${\mathfrak g}(1)$.
In fact, let ${\mathfrak n}$ (resp. ${\mathfrak n}^-$) be the Lie algebra of the unipotent radical of $B$ (resp. its opposite Borel subgroup); then ${\mathfrak g}={\mathfrak n}^-\oplus{\mathfrak t}\oplus{\mathfrak n}$.
Moreover ${\mathfrak n},{\mathfrak n}^-\subset{\mathfrak g}'\subset\sum_{i\in{\mathbb Z}}{\mathfrak g}(im/m')$, hence ${\mathfrak g}(1)\subset{\mathfrak t}$.
Suppose $h\in{\mathfrak g}(1)$.
Let $\alpha\in\Phi$: then $e_\alpha\in\sum_{i\in{\mathbb Z}}{\mathfrak g}(im/m')$.
But hence $[h,e_\alpha]=d\alpha(h)e_\alpha\in\sum_{i\in{\mathbb Z}}{\mathfrak g}(im/m')\cap \sum_{i\in{\mathbb Z}}{\mathfrak g}(im/m'+1)$.
Thus $d\alpha(h)=0$.
Since this is true for all $\alpha\in\Phi$, $h\in{\mathfrak z}({\mathfrak g})$.
\end{proof}

If $m=2$ and $T$ is a $\theta$-stable torus in $G$, then it is not difficult to see that there is a decomposition $T=T_+\cdot T_-$, where $T_+=\{ t\in T\,|\, \theta(t)=t\}^\circ$, $T_-=\{ t\in T\,|\, \theta(t)=t^{-1}\}^\circ$, and that the intersection is finite.
In fact, one also has a direct sum decomposition $\Lie(T)=\Lie(T_+)\oplus\Lie(T_-)$, hence the product map $T_+\times T_-\rightarrow T$ is a separable isogeny (see \cite[p. 290]{richinvs}).
Here we formulate a generalization of this result to arbitrary $m$.
For $d\geq 1$, denote by $p_d(x)$ the minimal polynomial over ${\mathbb Q}$ of a primitive $d$-th root of unity.
Since $p_d(x)$ has integer coefficients for each $d$, we can (and will) also consider $p_d(x)$ as a polynomial in ${\mathbb F}_p[x]$ or $k[x]$.
If $p\nmid d$, then $p_d(x)$ has no repeated roots in $k$.
If $T$ is a $\theta$-stable torus and $q(x)=\sum_{i=0}^n a_ix^i\in{\mathbb Z}[x]$ then we write $\overline{q}(\theta)$ for the rational endomorphism of $T$ defined by $t\mapsto\prod_{i=0}^n \theta^i(t)^{a_i}$.
The correspondence $q\mapsto \overline{q}(\theta)$ defines a homomorphism of rings ${\mathbb Z}[x]\rightarrow \End(T)$, $x\mapsto\theta$.
(Here the addition in $\End(T)$ is the pointwise product, and the multiplication is composition of endomorphisms.)

\begin{lemma}\label{stabletori}
Let $T$ be a $\theta$-stable torus in $G$ and  let ${\mathfrak t}=\Lie(T)$.
For each positive $d|m$ let $T_d=\{ t\in T\mid \overline{p_{m/d}}(\theta)(t)=e\}$.
(We count 1 as a divisor of $m$.)
Then $T_d$ is a subtorus of $T$, the intersection $T_{d_1}\cap T_{d_2}$ is finite for any distinct $d_1,d_2|m$, and $T=\prod_{d\mid m}T_d$ (almost direct product).

Moreover, ${\mathfrak t}=\sum^m_{i=1}{\mathfrak t}(i)$ (where ${\mathfrak t}(i)={\mathfrak t}\cap{\mathfrak g}(i)$) and $\sum_{(i,m)=d}{\mathfrak t}(i)=\Lie(T_d)$.
In particular, $T_1$ is the minimal subtorus of $T$ whose Lie algebra contains ${\mathfrak t}(1)$ and $\Lie(T_m)={\mathfrak t}(0)$.
\end{lemma}

\begin{proof}
Clearly $\Lie(T_d)\subseteq\{ t\in{\mathfrak t}\,|\, p_{m/d}(d\theta)(t)=0\}$ and hence $\Lie(T_{d_1}\cap T_{d_2})\subseteq\Lie(T_{d_1})\cap\Lie(T_{d_2})=\{ 0\}$ for $d_1\neq d_2$ since there exist $f,g\in k[t]$ such that $fp_{m/d_1}+gp_{m/d_2}=1$.
Thus $\Lie(T)\supset\oplus_{d|m}\Lie(T_d)$, $T_{d_1}\cap T_{d_2}$ is finite for $d_1\neq d_2$ and $T$ contains the almost direct product of the $T_{d}$.
For $d|m$, let $p'_{m/d}(x)=(x^m-1)/p_{m/d}(x)$.
Then $\overline{p_{m/d}}(\theta)\circ\overline{p'_{m/d}}(\theta)$ is trivial on $T$, hence $\overline{p'_{m/d}}(\theta)(T)\subseteq T_d$.
Moreover, $p'_{m/d}(d\theta)$ is bijective on $\Lie(T_d)$.
By dimensional considerations, $T_d=\overline{p'_d}(\theta)(T)$, $T=\prod_{d|m} T_d$ and ${\mathfrak t}=\oplus_{d|m}\Lie(T_d)$.
\end{proof}

Recall that if $\theta$ is an involution then a ($\theta$-stable) torus is called $\theta$-split or $\theta$-anisotropic if $\theta(t)=t^{-1}$ for all $t\in T$.
For $m>2$ we wish to distinguish two different cases:

\begin{definition}\label{thetasplitdef}
We say that a $\theta$-stable torus $S$ is {\bf $\theta$-split} if $S=S_1$, and is {\bf $\theta$-anisotropic} if $S_m=(S^\theta)^\circ$ is trivial.
(Hence any $\theta$-split torus is $\theta$-anisotropic.)
\end{definition}

We say that $\theta$ is {\bf of zero rank} if ${\mathfrak g}(1)$ contains no non-zero semisimple elements.

\begin{lemma}\label{zerorank}
If $p>2$, then the following are equivalent:

(i) ${\mathfrak g}(1)$ contains no non-central semisimple elements,

(ii) $\theta |_{G^{(1)}}$ is either of order less than $m$, or is of zero rank,

(iii) ${\mathfrak g}(1)={\mathfrak s}\oplus{\mathfrak n}$, where ${\mathfrak s}$ (resp. ${\mathfrak n}$) is the set of semisimple (resp. nilpotent) elements of ${\mathfrak g}(1)$ and ${\mathfrak s}\subseteq{\mathfrak z}({\mathfrak g})$.
\end{lemma}

\begin{proof}
We show first of all that (i) implies (ii).
If $\theta|_{G^{(1)}}$ is of order less than $m$ then all three conditions hold by Lemma \ref{triv}.
Hence suppose $\theta|_{G^{(1)}}$ is of order $m$.
Assume $G$ is semisimple; we will show that if ${\mathfrak g}(1)$ contains no non-central semisimple elements of ${\mathfrak g}$ then it contains no non-zero semisimple elements.
Let $\pi:\hat{G}\rightarrow G$ be the universal covering of $G$.
By Lemma \ref{cover}, there exists a unique lift $\hat\theta$ of $\theta$ to $\hat{G}$.
We claim that $\theta$ is of zero rank if and only if $\hat\theta$ is of zero rank.
Indeed, suppose ${\mathfrak c}\subseteq{\mathfrak g}(1)$ is a commutative subspace consisting of semisimple elements.
Let $T$ be a $\theta$-stable maximal torus of $L=Z_G({\mathfrak c})$.
Then ${\mathfrak c}\subseteq{\mathfrak z}({\mathfrak l})\subseteq{\mathfrak t}=\Lie(T)$.
(See \cite[Lemma 2.2]{comms} for the second inclusion.)
Let $T=\prod_{i|m}T_i$ be the decomposition of $T$ into subtori given by Lemma \ref{stabletori}.
Then ${\mathfrak c}\subseteq{\mathfrak t}(1)$ and hence $T_1$ is non-trivial.
Let $\hat{T}$ be the unique maximal torus of $\hat{G}$ such that $\pi(\hat{T})=T$.
Then $\hat{T}$ is $\hat\theta$-stable by uniqueness and there is a decomposition $\hat{T}=\prod_{i|d}\hat{T}_i$ into subtori $\hat{T}_i$ analogous to the $T_i$.
Moreover, it is easy to see from the proof of Lemma \ref{stabletori} that $\pi(\hat{T}_i)=T_i$.
Hence $\theta$ is of zero rank if and only if $\hat\theta$ is of zero rank.
Furthermore, it is well-known that $\ker d\pi\subseteq{\mathfrak z}(\hat{\mathfrak g})$.
Since $d\alpha(d\pi(h))=d\alpha(h)$ for any $h\in\Lie(\hat{T})$, it follows that if ${\mathfrak g}(1)$ contains no non-central semisimple elements of ${\mathfrak g}$ then $\hat{\mathfrak g}(1)$ contains no non-central semisimple elements of $\hat{\mathfrak g}$.
To prove that (i) implies (ii), we may therefore assume that $G$ is (semisimple and) simply-connected.

Since $G$ is the direct product of its minimal $\theta$-stable connected normal subgroups, we may assume $G$ is $\theta$-simple, that is, it has no non-trivial proper connected $\theta$-stable normal subgroups.
In this case $G=G_1\times G_2\times \ldots\times G_r$, where the $G_i$ are isomorphic almost simple (semisimple) groups, $\theta(G_i)=G_{i+1}$ ($1\leq i<r$) and $\theta(G_r)=G_1$.
(Thus $r\mid m$.)
It clearly changes nothing to replace $G$, $\theta$ and $m$ by $G_1$, $\theta^r$ and $m/r$: hence we may assume $G$ is simple.
(We may of course have $r=m$.
This reduces to `the $m=1$ case', which is just the adjoint action of $G$ on ${\mathfrak g}$.)
But now ${\mathfrak z}({\mathfrak g})$ is trivial unless $G$ is of type $A_{ip-1}$ for some $i$.
If $m=1$, then ${\mathfrak g}$ has some non-central semisimple elements by the assumption $p>2$.
If $m=2$ then there exists some non-trivial $\theta$-split torus $A\subset G$ by \cite[Prop. 1]{vust}.
Moreover, $Z_G(A)=Z_G(\Lie(A))$ by \cite[Lemma 2.4]{invs}.
Hence the assumption that there are no non-central semisimple elements in ${\mathfrak g}(1)$ implies that $m\geq 3$.
But now any automorphism of $\SL(ip)$ acts as either $(+1)$ or $(-1)$ on ${\mathfrak z}({\mathfrak g})$. Therefore ${\mathfrak g}(1)$ contains no non-zero semisimple elements.
Thus (i) implies that $\theta|_{G^{(1)}}$ is of zero rank.

To prove that (ii) implies (iii), we may assume once more that $\theta|_{G^{(1)}}$ has order $m$ by Lemma \ref{triv}.
By Lemma \ref{redcase} there is a $d\theta$-stable toral algebra ${\mathfrak h}\subset{\mathfrak g}$ such that ${\mathfrak g}={\mathfrak g}'\oplus{\mathfrak h}$ (where ${\mathfrak g}'=\Lie(G^{(1)})$).
Thus clearly ${\mathfrak g}(1)={\mathfrak g}'(1)\oplus{\mathfrak h}(1)$.
But ${\mathfrak g}'(1)$ consists of nilpotent elements, hence it remains only to show that ${\mathfrak h}(1)\subseteq{\mathfrak z}({\mathfrak g})$.
For this, let $T$ be a $\theta$-stable maximal torus of $G$ such that ${\mathfrak h}\subset{\mathfrak t}=\Lie(T)$, let $T'=T\cap G^{(1)}$ and let $Z=Z(G)^\circ$.
Since $T=T'\cdot Z$ and $\theta|_{G^{(1)}}$ is of zero rank, the kernel of the map $\overline{p'_1}(\theta):T\rightarrow T$ (see Lemma \ref{stabletori}) contains $T'$.
Hence $T_1$ is contained in $Z$.
It follows that ${\mathfrak t}(1)={\mathfrak s}\subseteq\Lie(Z)\subseteq{\mathfrak z}({\mathfrak g})$.
Since (iii) trivially implies (i), the proof of the lemma is complete.
\end{proof}

\begin{rk}\label{p=2}
The case of $m=1$, $G=\SL(2)$ gives a counter-example to Lemma \ref{zerorank} in characteristic 2.
However, from Sect. \ref{stand} onwards we will assume the standard hypotheses hold for $G$ (that $G^{(1)}$ is simply-connected and that there exists a non-degenerate $G$-equivariant symmetric bilinear form $\kappa:{\mathfrak g}\times{\mathfrak g}\rightarrow k$).
In these circumstances Lemma \ref{zerorank} then holds in characteristic 2.
This can be seen from the reduction theorem \ref{reduction} which allows us to restrict attention to the case $G=\tilde{G}$, hence to the cases $G=\SL(2m+1,k)$ or $G=\GL(2m,k)$.
The proof of Lemma \ref{zerorank} only falls down in characteristic 2 due to the possibility that all semisimple elements of ${\mathfrak g}$ are contained in the centre.
Hence this problem no longer occurs under the above assumptions.
On the other hand, we gain very little with this observation, since $G^{(1)}$ is then isomorphic to a product of groups of the form $\SL(V_i)$ and $\theta$ is inner.
\end{rk}

\begin{definition}
We say that $\theta$ is of {\bf zero semisimple rank} if the conditions of Lemma \ref{zerorank} hold.
\end{definition}

Finally, we state the following slightly modified version of \cite[Lemma 1.4(v)]{invs} for use in Sect. 3.

\begin{lemma}\label{GLnautos}
Let $G=\GL(n,k)$, ${\mathfrak g}=\Lie(G)$, ${\mathfrak g}'=\Lie(G^{(1)})$, where $p|n$.
Denote by $\Aut G$ (resp. $\Aut{\mathfrak g}$) the (abstract) group of algebraic (resp. restricted Lie algebra) automorphisms of $G$ (resp. ${\mathfrak g}$).
%
%
%
%
%
If $\eta$ is an automorphism of ${\mathfrak g}'$ of order $m$, $p\nmid m$ then there is a unique $\theta\in\Aut G$ (resp. $\psi\in\Aut{\mathfrak g}$) of order $m$ such that $d\theta |_{{\mathfrak g}'}=\eta$ (resp. $\psi |_{{\mathfrak g}'}=\eta$).
\end{lemma}

\begin{proof}
Although one assumes $p\neq 2$ in \cite{invs} this is not used in the proof of \cite[Lemma 1.4]{invs}.
In particular, $\Aut {\mathfrak g}\cong\Aut{\mathfrak g}'\times \mu_p$ and hence there exists a unique automorphism of ${\mathfrak g}$ of order $m$ whose restriction to ${\mathfrak g}'$ is $\eta$ \cite[Lemma 1.4(iv)]{invs}.
On the other hand, $\Aut G\cong\Aut G^{(1)}$ (by restriction) unless $n=2$, in which case the kernel of the natural map $\Aut G\rightarrow\Aut G^{(1)}$ is of order 2 \cite[Lemma 1.4(ii)]{invs}.
Since differentiation $d:\Aut G^{(1)}\rightarrow\Aut{\mathfrak g}'$ is bijective \cite[Lemma 1.4(iii)]{invs} this completes the proof.
\end{proof}

\section{Cartan subspaces}

\begin{definition}
A subspace ${\mathfrak c}$ of ${\mathfrak g}(1)$ is a {\bf Cartan subspace} if it is maximal among the commutative subspaces of ${\mathfrak g}(1)$ consisting of semisimple elements.
\end{definition}

If $m=2$ then a Cartan subspace ${\mathfrak c}$ of ${\mathfrak g}(1)$ satisfies ${\mathfrak z}_{{\mathfrak g}(1)}({\mathfrak c})={\mathfrak c}$.
For $m>2$ this may no longer hold.
(This can already be seen in the zero rank case.)
Recall that if ${\mathfrak h}\subset{\mathfrak g}$ is a nilpotent subalgebra then there is a {\bf Fitting decomposition} ${\mathfrak g}={\mathfrak g}^0({\mathfrak h})\oplus{\mathfrak g}^1({\mathfrak h})$, where $\ad{\mathfrak h}$ acts nilpotently on ${\mathfrak g}^0({\mathfrak h})$ and all weights of ${\mathfrak h}$ on ${\mathfrak g}^1({\mathfrak h})$ are non-zero.
There is an open subset $U$ of ${\mathfrak h}$ such that for $x\in U$, $\ad x$ is nilpotent on ${\mathfrak g}^0({\mathfrak h})$ and invertible on ${\mathfrak g}^1({\mathfrak h})$, that is, ${\mathfrak g}^i({\mathfrak h})={\mathfrak g}^i(kx)$ for $i=0,1$.
The following lemma is a slight modification of \cite[Lemma 1]{kostrall}.

\begin{lemma}\label{fitting}
Let ${\mathfrak h}\subseteq{\mathfrak g}(1)$ be a commutative subspace.
Then ${\mathfrak g}(1)={\mathfrak g}^0({\mathfrak h})\cap{\mathfrak g}(1)\oplus{\mathfrak g}^1({\mathfrak h})\cap{\mathfrak g}(1)$.
\end{lemma}

\begin{proof}
Let $x\in U$, where $U$ is the set defined in the paragraph above.
Since $(\ad x)$ acts invertibly (resp. nilpotently) on ${\mathfrak g}^1({\mathfrak h})$ (resp. ${\mathfrak g}^0({\mathfrak h})$), so does $(\ad x)^m$.
But $(\ad x)^m({\mathfrak g}(i))\subset{\mathfrak g}(i)$ for each $i\in{\mathbb Z}/m{\mathbb Z}$.
\end{proof}

If ${\mathfrak h}\subseteq{\mathfrak g}(1)$ is a commutative subspace then write ${\mathfrak g}^i({\mathfrak h})(1)$ for ${\mathfrak g}^i({\mathfrak h})\cap{\mathfrak g}(1)$.
Lemma \ref{fitting} allows us to prove the following lemma by a standard argument.

\begin{lemma}\label{sep1}
Let ${\mathfrak h}\subset{\mathfrak g}(1)$ be commutative.
Then the morphism $\phi:G(0)\times{\mathfrak g}^0({\mathfrak h})(1)\rightarrow{\mathfrak g}(1)$, $(g,x)\rightarrow\Ad g(x)$ is dominant and separable.
\end{lemma}

\begin{proof}
Let $h\in{\mathfrak h}$ be such that ${\mathfrak g}^0(kh)={\mathfrak g}^0({\mathfrak h})$ and ${\mathfrak g}^1(kh)={\mathfrak g}^1({\mathfrak h})$.
We claim that $d\phi_{(e,h)}$ is surjective.
Indeed, identifying $T_{(e,h)}(G(0)\times{\mathfrak g}^0({\mathfrak h})(1))$ with ${\mathfrak g}(0)\oplus{\mathfrak g}^0({\mathfrak h})(1)$ in the natural way, it can easily be seen that $d\phi_{(e,h)}(x,y)=[x,h]+y$.
Hence $\im d\phi_{(e,h)}=[{\mathfrak g}(0),h]+{\mathfrak g}^0({\mathfrak h})(1)$.
But $[{\mathfrak g}(0),h]=[{\mathfrak g},h]\cap{\mathfrak g}(1)\supset{\mathfrak g}^1({\mathfrak h})(1)$, thus $d\phi_{(e,h)}$ is surjective.
By \cite[AG. 17.3]{borel}, $\phi$ is dominant and separable.
\end{proof}

\begin{corollary}\label{sep}
Let ${\mathfrak c}$ be a Cartan subspace of ${\mathfrak g}(1)$.
Then the morphism $G(0)\times{\mathfrak z}_{{\mathfrak g}(1)}({\mathfrak c})\rightarrow{\mathfrak g}(1)$, $(g,x)\rightarrow\Ad g(x)$ is dominant and separable.
\end{corollary}

\begin{proof}
Since ${\mathfrak g}$ is a completely reducible $\ad{\mathfrak c}$-module, ${\mathfrak z}_{{\mathfrak g}(1)}({\mathfrak c})={\mathfrak g}^0({\mathfrak c})(1)$.
Hence we can apply Lemma \ref{sep1}.
\end{proof}

Recall \cite[\S 3]{vin} that $c\in{\mathfrak c}$ is {\it an element in general position} if ${\mathfrak z}_{\mathfrak g}(c)={\mathfrak z}_{\mathfrak g}({\mathfrak c})$.
In common with \cite{vin}, denote by $R({\mathfrak c})$ the set of $x\in{\mathfrak z}_{{\mathfrak g}(1)}({\mathfrak c})$ such that the semisimple part of $x$ (necessarily in ${\mathfrak c}$) is an element in general position.

\begin{theorem}\label{carts}
Any two Cartan subspaces of ${\mathfrak g}(1)$ are conjugate by an element of $G(0)$.
\end{theorem}

\begin{proof}
Let ${\mathfrak c}_1$ and ${\mathfrak c}_2$ be two Cartan subspaces.
By Lemma \ref{sep}, $G(0)\cdot{\mathfrak z}_{{\mathfrak g}(1)}({\mathfrak c}_1)$ and $G(0)\cdot{\mathfrak z}_{{\mathfrak g}(1)}({\mathfrak c}_2)$ are dense constructible (that is, unions of locally closed) subsets of ${\mathfrak g}(1)$.
But hence $(G(0)\cdot{\mathfrak z}_{{\mathfrak g}(1)}({\mathfrak c}_1))\cap{\mathfrak z}_{{\mathfrak g}(1)}({\mathfrak c}_2)$ contains a non-empty open subset of ${\mathfrak z}_{{\mathfrak g}(1)}({\mathfrak c}_2)$.
Since $R({\mathfrak c}_2)$ is dense in ${\mathfrak z}_{{\mathfrak g}(1)}({\mathfrak c}_2)$, it intersects non-trivially with $G(0)\cdot{\mathfrak z}_{{\mathfrak g}(1)}({\mathfrak c}_1)$.
But for any $x\in R({\mathfrak c}_2)$, ${\mathfrak c}_2$ is the set of semisimple elements of ${\mathfrak z}_{{\mathfrak g}(1)}(x)$.
Hence if $\Ad g^{-1}(x)\in {\mathfrak z}_{{\mathfrak g}(1)}({\mathfrak c}_1)$ then clearly $\Ad g^{-1}({\mathfrak c}_2)={\mathfrak c}_1$.
\end{proof}

\begin{corollary}\label{semicart}
Let ${\mathfrak c}$ be a Cartan subspace of ${\mathfrak g}(1)$.
Then any semisimple element of ${\mathfrak g}(1)$ is conjugate to an element of ${\mathfrak c}$.
\end{corollary}

Note that if $p>2$ or $G$ satisfies the standard hypotheses then ${\mathfrak z}_{{\mathfrak g}(1)}({\mathfrak c})={\mathfrak c}\oplus{\mathfrak u}$ for a subspace ${\mathfrak u}$ consisting of nilpotent elements by Lemma \ref{zerorank}.
The assumption $p>2$ is not required for Thm. \ref{carts} to hold.
There is a natural relationship between Cartan subspaces of ${\mathfrak g}(1)$ and maximal $\theta$-split tori in $G$.
Denote by $\varphi(m)$ the Euler number of $m$.

\begin{lemma}\label{thetasplit}
Let ${\mathfrak c}$ be a Cartan subspace of ${\mathfrak g}(1)$.
Then there exists a maximal $\theta$-split torus $T_1$ such that $\Lie(T_1)\supset{\mathfrak c}$.
Moreover, $\dim T_1=\dim{\mathfrak c}\cdot\varphi(m)$ and $T_1$ is a minimal torus in $G$ such that ${\mathfrak c}\subset\Lie(T_1)$.
If $p>2$ or if $G$ satisfies the standard hypotheses, then $T_1$ is unique.
\end{lemma}

\begin{proof}
Let $L=Z_G({\mathfrak c})$, and let $T$ be any $\theta$-stable maximal torus of $L$.
Then ${\mathfrak c}\subseteq{\mathfrak z}({\mathfrak l})\subseteq\Lie(T)$.
Let $T_d$, $d|m$ be the subtori of $T$ given by Lemma \ref{stabletori}.
Then ${\mathfrak c}\subset\Lie(T_1)$ and hence by maximality ${\mathfrak c}=\Lie(T_1)(1)$.
But if $T_1$ is properly contained in a $\theta$-split torus of $G$ then ${\mathfrak c}$ cannot be a Cartan subspace, hence $T_1$ is maximal.
This proves the first statement of the lemma.
The second follows from Lemma \ref{stabletori}.
For the final assertion, $\theta|_{L^{(1)}}$ is of zero rank, hence $T_1\subseteq Z(L)^\circ$ by Lemma \ref{zerorank} and Remark \ref{p=2}.
But therefore $T_1$ is the unique maximal $\theta$-split torus of $L$.
\end{proof}

The following is an analogue of \cite[11.1]{richinvs}.
The proof is essentially identical; we include it for the reader's convenience.

\begin{lemma}\label{conj}
Let ${\mathfrak c}$ be a Cartan subspace of ${\mathfrak g}(1)$ and let $Y\subset{\mathfrak c}$ be a subset of ${\mathfrak c}$.
If $g\in G(0)$ is such that $\Ad g(Y)\subset{\mathfrak c}$ then there exists $n\in N_{G(0)}({\mathfrak c})$ such that $\Ad n(y)=\Ad g(y)$ for all $y\in Y$.
\end{lemma}

\begin{proof}
Let $L=Z_G(\Ad g(Y))$ and ${\mathfrak l}=\Lie(L)={\mathfrak z}_{\mathfrak g}(\Ad g(Y))$.
Then $L$ is a $\theta$-stable Levi subgroup of $G$ and ${\mathfrak c},\Ad g({\mathfrak c})$ are Cartan subspaces of ${\mathfrak l}(1)$.
Therefore we can apply Thm. \ref{carts}.
Thus there is $h\in L(0)\subseteq L\cap G(0)$ such that $\Ad hg({\mathfrak c})={\mathfrak c}$.
But $\Ad hg(y)=\Ad g(y)$ for all $y\in Y$.
\end{proof}

\begin{rk}
It is clear from the above proof that the Lemma is valid on replacing $G(0)$ and $N_{G(0)}({\mathfrak c})$ by $G^\theta$ and $N_{G^\theta}({\mathfrak c})$ (resp. $G_Z^\theta=\{ g\in G\mid g^{-1}\theta(g)\in Z(G)\}$ and $N_{G_Z^\theta}({\mathfrak c})$).
\end{rk}

We now consider properties of $G(0)$-orbits in ${\mathfrak g}(1)$.
Let $V$ be a finite-dimensional $k$-vector space and let $H$ be an affine algebraic group which acts linearly (and rationally) on $V$.
Denote by $h\cdot v$ the action of $h\in H$ on $v\in V$, and similarly $H\cdot v$ the $H$-orbit through $v$.
For a subset $Y$ of an affine variety $X$, let $\overline{Y}$ denote the Zariski closure of $Y$ in $X$.
(The variety $X$ will always be clear from the context.
In particular, all closures will be in ${\mathfrak g}(1)$ unless otherwise specified.)
Recall that $v\in V$ is {\it unstable} if $0\in\overline{H\cdot v}$.

\begin{lemma}\label{unstable}
Let $x\in{\mathfrak g}(1)$ be nilpotent.
Then $x$ is $G(0)$-unstable.
\end{lemma}

\begin{proof}
We apply Kawanaka's theorem \cite{kawanaka} on nilpotent $G(0)$-orbits in ${\mathfrak g}(1)$: for any nilpotent element $x\in{\mathfrak g}(1)$ there is a fundamental pair $(B,T)$ for $\theta$ and a $W(G,T)$-conjugate $h$ of some weighted Dynkin diagram $h_+$ over $T$ (in the Bala-Carter classification) such that $x$ is in the nilpotent $G$-conjugacy class corresponding to $h_+$ and $x\in{\mathfrak g}(2;h)$.
(See \cite{kawanaka} for further details.)
Moreover \cite[Def. 3.1.1]{kawanaka} $h$ is $\theta$-stable.
But for any such weighted Dynkin diagram there is some positive integer $l$ and a cocharacter $\lambda:k^\times\rightarrow T$ such that $\langle\lambda,\alpha\rangle= lh(\alpha)$ for all $\alpha\in\Phi(G,T)$.
Then $0\in\overline{\{ \Ad \lambda(t)(x):t\in k^\times\}}$.
\end{proof}

\begin{rk}\label{associated}
Let $x\in{\mathfrak g}$ be any nilpotent element.
Then a cocharacter $\lambda:k^\times\rightarrow G$ is an {\it associated cocharacter} for $x$ if:

 - $\Ad\lambda(t)(x)=t^2x$ for all $t\in k^\times$, that is, $x\in{\mathfrak g}(2;\lambda)$,
 
 - ${\mathfrak z}_{\mathfrak g}(x)\subseteq\sum_{i\geq 0}{\mathfrak g}(i;\lambda)$,
 
 - There exists a Levi subgroup $L$ of $G$ such that $\lambda(k^\times)\subset L^{(1)}$ and $e$ is a distinguished nilpotent element of $\Lie(L)$.

According to the Bala-Carter-Pommerening theorem (see \cite{premorbits} for a recent proof) any nilpotent orbit in ${\mathfrak g}$ has an associated cocharacter.
If $e\in{\mathfrak g}(1)$ then the argument in \cite[Cor. 5.4]{invs} shows that there exists a cocharacter $\lambda:k^\times\rightarrow G(0)$ which is associated to $e$.
(Moreover, any two such are conjugate by an element of $Z_{G(0)}(e)^\circ$.)
Thus in the proof above we can choose $l$ to be equal to 1.
\end{rk}

For the following lemma we essentially follow Vinberg's proof in characteristic zero \cite[1.3-4]{vin}.

\begin{lemma}\label{closed}
Let $x\in{\mathfrak g}(1)$ be semisimple.
Then each irreducible component of $G\cdot x\cap{\mathfrak g}(1)$ is a single $G(0)$-orbit and, conversely, each $G(0)$-orbit in $G\cdot x\cap{\mathfrak g}(1)$ is an irreducible component.
Hence all semisimple $G(0)$-orbits in ${\mathfrak g}(1)$ are closed.
\end{lemma}

\begin{proof}
It is well-known that $G\cdot x$ is closed for any semisimple $x$, hence $G\cdot x\cap{\mathfrak g}(1)$ is closed for $x\in{\mathfrak g}(1)$ semisimple.
Since $x$ is semisimple, $T_x(G\cdot x)=[{\mathfrak g},x]$ (making the obvious identifications) and hence $T_x(G\cdot x\cap{\mathfrak g}(1))\subseteq[{\mathfrak g},x]\cap{\mathfrak g}(1)=[{\mathfrak g}(0),x]$.
On the other hand, ${\mathfrak z}_{{\mathfrak g}(0)}(x)={\mathfrak z}_{\mathfrak g}(x)\cap{\mathfrak g}(0)=\Lie((Z_G(x)^\theta)$ and hence $T_x(G(0)\cdot x)=[{\mathfrak g}(0),x]$ by equality of dimensions.
But clearly $T_x(G\cdot x\cap{\mathfrak g}(1))\supseteq T_x(G(0)\cdot x)$, hence equality holds.
It follows that the dimension of any irreducible component of $G\cdot x\cap {\mathfrak g}(1)$ containing $x$ is at most $\dim G(0)\cdot x$.
Thus $x$ is a smooth point of $G\cdot x\cap{\mathfrak g}(1)$ and (therefore) $G(0)\cdot x$ is the unique irreducible component of $G\cdot x\cap {\mathfrak g}(1)$ containing $x$.
\end{proof}

\begin{corollary}\label{closedcor}
Let $x\in{\mathfrak g}(1)$.
Then $G(0)\cdot x_s$ is the unique closed orbit in $\overline{G(0)\cdot x}$.
\end{corollary}

\begin{proof}
Let $L=Z_G(x_s)$.
Then $x_n\in\Lie(L)\cap {\mathfrak g}(1)$ and hence by Lemma \ref{unstable}, the closure $\overline{(L\cap G(0))\cdot x_n}$ contains $0$.
It follows that $x_s\in\overline{(L\cap G(0))\cdot x}\subseteq \overline{G(0)\cdot x}$.
Moreover, $G(0)\cdot x_s$ is closed by Lemma \ref{closed}.
But it is well-known that $\overline{G(0)\cdot x}$ contains a unique closed $G(0)$-orbit (see for example \cite[8.3]{hum}).
\end{proof}

We briefly recall the basic definition and properties of the categorical quotient.
Let $H$ be an affine algebraic group such that $H^\circ$ is reductive (possibly trivial) and let $X$ be an affine variety.
We say that $H$ acts {\it morphically} on $X$ if $H$ acts on $X$, and the corresponding map $H\times X\rightarrow X$ is a morphism of varieties.
The ring of invariants $k[X]^H$ is finitely generated.
The corresponding affine variety $\Spec (k[X]^H)$ is the {\it categorical quotient} of $X$ by $H$ and the morphism $\pi=\pi_{X,H}:X\rightarrow X\quot H$ induced by the algebra embedding $k[X]^H\hookrightarrow k[X]$ is the {\it quotient morphism}.
We have the following well-known properties:

 - $\pi$ is surjective,
 
 - If $U_1,U_2$ are disjoint $H$-stable closed subsets of $X$ then there exists $f\in k[X]^H$ such that $f(x)=0$ for all $x\in U_1$ and $f(x)=1$ for all $x\in U_2$.

 - Each fibre $\pi^{-1}(\xi)$ is a finite union of $H$-orbits and contains a unique closed $H$-orbit, which we denote $T(\xi)$, and which is also the unique orbit of minimal dimension.
 
 - For $x\in X$ and $\xi\in X\quot H$, $\pi(x)=\xi$ if and only if $\overline{H\cdot x}\supseteq T(\xi)$.
 
 - If $X$ is normal, then so is $X\quot G$.

In the present circumstances we are interested in the quotient ${\mathfrak g}(1)\quot G(0)$.
The closed orbits in ${\mathfrak g}(1)$ are precisely the semisimple orbits (Lemma \ref{closed}) and each semisimple orbit meets ${\mathfrak c}$ (Cor. \ref{semicart}).
Furthermore, two elements of ${\mathfrak c}$ are conjugate by an element of $G(0)$ if and only if they are conjugate by an element of $N_{G(0)}({\mathfrak c})$ by Lemma \ref{conj}.
Let $W_{\mathfrak c}=N_{G(0)}({\mathfrak c})/Z_{G(0)}({\mathfrak c})$.
Hence the embedding $j:{\mathfrak c}\hookrightarrow{\mathfrak g}(1)$ induces a bijective morphism $j':{\mathfrak c}/W_{\mathfrak c}\rightarrow{\mathfrak g}(1)\quot G(0)$ such that the following diagram is commutative:
\begin{diagram} {\mathfrak c} & \rTo^j & {\mathfrak g}(1) \\
\dTo & & \dTo \\
{\mathfrak c}/W_{\mathfrak c} & \rTo^{j'} & {\mathfrak g}(1)\quot G(0) \end{diagram}
Note that $\pi_{\mathfrak c}:{\mathfrak c}\rightarrow{\mathfrak c}\quot W_{\mathfrak c}$ and $j'$ are finite morphisms and hence their composition maps open sets to open sets \cite[4.2]{hum}.
Since the set of elements of ${\mathfrak c}$ in general position is clearly open, its image $\pi_{{\mathfrak g}(1)}(R({\mathfrak c}))$ is also open in ${\mathfrak g}(1)\quot G(0)$.
Thus we have proved:

\begin{lemma}
$G(0)\cdot R({\mathfrak c})$ is open in ${\mathfrak g}(1)$.
\end{lemma}

We recall that the quotient morphism is not in general separable, even if $X$ is a vector space \cite{martin-neeman}.
The present case poses some difficulties, since a commonly used criterion for separability \cite[9.3]{richinvs} does not apply.
The following result, which appeared in \cite{vin} in the case of characteristic zero, provides the solution.

\begin{lemma}\label{rationals}
Assume $p>2$ or that $G$ satisfies the standard hypotheses.
Then $k({\mathfrak g}(1))^{G(0)}=\Frac (k[{\mathfrak g}(1)]^{G(0)})$.
\end{lemma}

\begin{proof}
As above, let $R({\mathfrak c})$ denote the set of $x\in{\mathfrak z}_{{\mathfrak g}(1)}(c)$ such that the semisimple part of $x$ is an element in general position in ${\mathfrak c}$.
Let $L=Z_G({\mathfrak c})$, a $\theta$-stable Levi subgroup of $G$, and let ${\mathfrak l}=\Lie(L)={\mathfrak z}_{\mathfrak g}({\mathfrak c})$, $L(0)=(L^\theta)^\circ$.
By Lemma \ref{zerorank}, ${\mathfrak l}(1)={\mathfrak c}\oplus{\mathfrak u}$, where ${\mathfrak u}$ is the set of nilpotent elements of ${\mathfrak l}(1)$.
Let $h\in k({\mathfrak g}(1))^{G(0)}$.
Then the domain $\dom h$ of $h$ is open in ${\mathfrak g}(1)$ and hence intersects non-trivially with $G(0)\cdot R({\mathfrak c})$.
By $G(0)$-invariance, it intersects non-trivially with $R({\mathfrak c})$, hence $h$ restricts to a rational function $h|_{{\mathfrak l}(1)}\in k({\mathfrak l}(1))^{L(0)}$.
Let $c+u\in R({\mathfrak c})$, where $c\in{\mathfrak c},u\in{\mathfrak u}$.
We claim that $c+u\in\dom h$ if and only if $c\in\dom h$.
Indeed, it will suffice to show that the equivalence holds in ${\mathfrak l}(1)$: for if $x\in\dom h|_{{\mathfrak l}(1)}$ then $h|_{{\mathfrak l}(1)}$ is defined on an open neighbourhood of $x$ in ${\mathfrak l}(1)$, and hence by Cor. \ref{sep} and $G(0)$-invariance, $h$ is defined on an open neighbourhood of $x$ in ${\mathfrak g}(1)$.

Thus, after replacing $G$ by $L$, we may assume that $\theta$ is of zero semisimple rank.
Since $k[{\mathfrak c}\oplus{\mathfrak u}]\cong k[{\mathfrak c}]\otimes k[{\mathfrak u}]$ is a unique factorization domain, we can write $h$ as $f/g$, where $f$ and $g$ are coprime polynomials.
This expression is unique up to non-zero scalar multiplication of $f$ and $g$, and $h$ is $G(0)$-invariant, thus for each $x\in G(0)$, $x\cdot f=\xi f$ and $x\cdot g=\xi g$ for some $\xi\in k^\times$.
Hence there is a rational homomorphism $\rho:G(0)\rightarrow k^\times$ such that $x\cdot f=\rho(x)f$ and $x\cdot g=\rho(x)g$ for all $x\in G(0)$.
It follows from Lemma \ref{unstable} that there are finitely many $G(0)$-orbits in ${\mathfrak u}$, hence that there is a unique dense orbit ${\cal O}$.
But therefore any $G(0)$-invariant rational function on ${\mathfrak u}$ is constant on ${\cal O}$, and therefore constant on all of ${\mathfrak u}$.
Thus $k[{\mathfrak u}]_\rho=\{ p\in k[{\mathfrak u}]\mid x\cdot p=\rho(x) p\mbox{ for all }x\in G(0)\}$ is of dimension (at most) 1.
(If not, then we can find $p,q\in k[{\mathfrak u}]_\rho$ such that $p/q$ is non-constant.)
Let $\sum_{i=1}^n f_i^{(1)}\otimes f_i^{(2)}$ be an expression for $f$ with $n$ minimal, where $f_i^{(1)}\in k[{\mathfrak c}]$ and $f_i^{(2)}\in k[{\mathfrak u}]$.
Then $x\cdot f=\sum_{i=1}^n f_i^{(1)}\otimes (x\cdot f_i^{(2)})=\rho(x) f$.
Since the $f_i^{(1)}$ are linearly independent, it follows that $(x\cdot f_i^{(2)})(u)=\rho(x) f_i^{(2)}(u)$ for each $u\in {\mathfrak u}$ and each $i$, $1\leq i\leq n$, hence $f_i^{(2)}\in k[{\mathfrak u}]_\rho$.
By minimality, $f=f^{(1)}\otimes f^{(2)}$ for some $f^{(1)}\in k[{\mathfrak c}]$, $f^{(2)}\in k[{\mathfrak u}]_\rho$.
But now by the same argument $g=g^{(1)}\otimes g^{(2)}$ with $g^{(1)}\in k[{\mathfrak c}]$ and $g^{(2)}\in k[{\mathfrak u}]_\rho$, thus $g^{(2)}$ is a scalar multiple of $f^{(2)}$.
By coprimeness, $f,g\in k[{\mathfrak c}]$.
It follows immediately that $c+u\in\dom h$ if and only if $c\in\dom h$.

Returning to the general case (where $\theta$ is not no longer of zero semisimple rank), we can now apply the argument of \cite[9.3]{richinvs}.
Hence let $X_1={\mathfrak g}(1)\setminus (G(0)\cdot R({\mathfrak c}))$.
Thus $Y=X_1\cup X\setminus \dom h$ is closed in $X$.
Let $x\in X\setminus Y=(G(0)\cdot R({\mathfrak c}))\cap\dom h$.
Then the semisimple part of $x$ is $G(0)$-conjugate to an element $c$ of ${\mathfrak c}$ in general position.
Let $U=G(0)\cdot (c+{\mathfrak u})$.
Since $U=\pi^{-1}(\pi(c))$ by Cor. \ref{closedcor}, $U$ is closed (and $G(0)$-stable) in ${\mathfrak g}(1)$.
Moreover, $h$ is defined at each point of $U$ and hence $U\cap Y=\emptyset$.
Thus there exists $g\in k[{\mathfrak g}(1)]^{G(0)}$ such that $g(u)=1$ for all $u\in U$ and $g(y)=0$ for all $y\in Y$.
In particular, $\dom h$ contains $X_g=\{Êx\in X\mid g(x)\neq 0\}$.
It follows that $h=f/g^r$ for some $r\geq 0$ and some $f\in k[{\mathfrak g}(1)]$.
Hence $h\in \Frac k[{\mathfrak g}(1)]^{G(0)}$.
\end{proof}

\begin{rk}
In the proof above we showed that restriction $k({\mathfrak g}(1))^{G(0)}\rightarrow k({\mathfrak z}_{\mathfrak g}({\mathfrak c}))$ induces a well-defined injective homomorphism from $k({\mathfrak g}(1))^{G(0)}$ to $k({\mathfrak c})$, hence that $k({\mathfrak g}(1))^{G(0)}$ is a subfield of $k({\mathfrak c})^W=\Frac k[{\mathfrak c}]^W$.
One can check that this also holds if $m=1$, $G=\SL(2)$ and $\charac k=2$.
However, the Chevalley restriction theorem (Thm. \ref{chev} below) does not hold in this case.
In fact, in this case the quotient morphism is separable, but the induced morphism ${\mathfrak c}={\mathfrak c}/W\rightarrow{\mathfrak g}(1)\quot G(0)$ is purely inseparable.
This comes down to the fact that the natural morphism ${\mathfrak c}\times{\cal N}({\mathfrak z}_{{\mathfrak g}(1)}({\mathfrak c}))\rightarrow{\mathfrak z}_{{\mathfrak g}(1)}({\mathfrak c})$ is inseparable.
\end{rk}

\begin{corollary}\label{sepquot}
If $p>2$ or if $G$ satisfies the standard hypotheses then the quotient morphism ${\mathfrak g}(1)\rightarrow {\mathfrak g}(1)\quot G(0)$ is separable.
\end{corollary}

\begin{proof}
By \cite[\S AG. 2.4]{borel} the field extension $k({\mathfrak g}(1))\supset k({\mathfrak g}(1))^{G(0)}$ is separable.
Hence we apply Lemma \ref{rationals}.
\end{proof}

This preparation allows us to prove the following form of the Chevalley Restriction Theorem.
Our proof follows \cite[11.3]{richinvs}.

\begin{theorem}\label{chev}
Suppose $p>2$ or $G$ satisfies the standard hypotheses.
Then the embedding $j:{\mathfrak c}\hookrightarrow{\mathfrak g}(1)$ induces an isomorphism of varieties $j':{\mathfrak c}/W_{\mathfrak c}\rightarrow {\mathfrak g}(1)\quot G(0)$.
\end{theorem}

\begin{proof}
As remarked above, $j'$ is bijective.
Since ${\mathfrak c}/W_{\mathfrak c}$ and ${\mathfrak g}(1)\quot G(0)$ are normal by a standard fact about the categorical quotient, it will suffice to show that $j'$ is separable.
Let $L=Z_G({\mathfrak c})$, a $\theta$-stable Levi subgroup of $G$, and let ${\mathfrak l}=\Lie(L)={\mathfrak z}_{\mathfrak g}({\mathfrak c})={\mathfrak c}\oplus{\mathfrak u}$, where ${\mathfrak u}$ is the set of nilpotent elements of ${\mathfrak g}(1)$ commuting with ${\mathfrak c}$.
By Cor. \ref{sepquot} the quotient morphism $\pi:{\mathfrak g}(1)\rightarrow{\mathfrak g}(1)\quot G(0)$ is separable.
Moreover $\phi:G(0)\times{\mathfrak l}(1)\rightarrow{\mathfrak g}(1)$, $(g,x)\mapsto\Ad g(x)$, is a separable morphism by Cor. \ref{sep}.
Applying the argument in \cite[11.3]{richinvs}, the induced morphism ${\mathfrak l}(1)\rightarrow{\mathfrak g}(1)\quot G(0)$ is separable.
Since $L(0)$ acts trivially on ${\mathfrak c}$, $k[{\mathfrak l}(1)]^{L(0)}=k[{\mathfrak c}\oplus{\mathfrak u}]^{L(0)}=k[{\mathfrak c}]$.
Hence the composition $\sigma$ of the embedding ${\mathfrak c}\rightarrow{\mathfrak l}(1)$ with the quotient morphism ${\mathfrak l}(1)\rightarrow{\mathfrak l}(1)\quot L(0)$ is a $N_{G(0)}({\mathfrak c})$-equivariant isomorphism of varieties.
It follows that there is an isomorphism $\overline\sigma$ making the following diagram commutative:
\begin{diagram}
{\mathfrak c} & \rTo^\sigma & {\mathfrak l}(1)\quot L(0) \\
\dTo & & \dTo \\
{\mathfrak c}/W_{\mathfrak c} & \rTo^{\overline\sigma} & {\mathfrak l}(1)\quot N_{G(0)}({\mathfrak c})
\end{diagram}
On the other hand, the separable morphism ${\mathfrak l}(1)\rightarrow{\mathfrak g}(1)\quot G(0)$ clearly factors through ${\mathfrak l}(1)\rightarrow{\mathfrak l}(1)\quot N_{G(0)}({\mathfrak c})$, and hence ${\mathfrak l}(1)\quot N_{G(0)}({\mathfrak c})\rightarrow {\mathfrak g}(1)\quot G(0)$ is separable.
Thus $j'$ is separable.
By \cite[\S AG. 18.2]{borel}, $j'$ is an isomorphism.
\end{proof}

\begin{rk}
In the case $m=2$, Thm. \ref{chev} appeared in \cite{invs}, but with the requirement that $G$ satisfy the standard hypotheses.
\end{rk}

Following \cite[Cor. 2 to Thm. 4]{vin}, we have:

\begin{corollary}\label{dimcor}
The fibres of $\pi:{\mathfrak g}(1)\rightarrow{\mathfrak g}(1)\quot G(0)$ are equidimensional, of dimension $\dim{\mathfrak g}(1)-r$.
\end{corollary}

\begin{proof}
By standard facts on morphisms (see for example \cite[4.1,4.3]{hum}) each irreducible component of each fibre of $\pi$ has dimension at least $\dim{\mathfrak g}(1)-\dim{\mathfrak g}(1)\quot G(0)=\dim{\mathfrak g}(1)-r$, and there exists an open subset $U$ of ${\mathfrak g}(1)\quot G$ such that the fibre $\pi^{-1}(u)$ is of pure dimension ${\mathfrak g}(1)-r$ for each $u\in U$.
Let $q=\min_{x\in{\mathfrak g}(1)}\dim Z_{G(0)}(x)$.
Then the set $\{ x\in{\mathfrak g}(1)\mid \dim Z_{G(0)}(x)=q\}$ is open in ${\mathfrak g}(1)$ and hence intersects non-trivially with $\pi^{-1}(U)$.
Since each irreducible component of each fibre of $\pi$ has an open orbit, we have $\dim G(0)-q=\dim{\mathfrak g}(1)-r$.
But therefore the fibres of $\pi$ are all of pure codimension $r$ in ${\mathfrak g}(1)$.
\end{proof}

\section{A $\theta$-stable reduction}\label{stand}

Until now we have made only one assumption on $G$:

(A) $p$ is good.

From now on we make the additional assumptions:

(B) $G^{(1)}$ is simply-connected.

(C) There is a non-degenerate symmetric bilinear $G$-equivariant form $\kappa:{\mathfrak g}\times{\mathfrak g}\rightarrow k$.

Let $G_1,G_2,\ldots ,G_r$ be the minimal normal subgroups of $G^{(1)}$ and let ${\mathfrak g}_i=\Lie(G_i)$.
Hence $G^{(1)}=G_1\times \ldots \times G_r$ and ${\mathfrak g}'=\Lie(G^{(1)})={\mathfrak g}_1\oplus\ldots\oplus{\mathfrak g}_r$.
Define groups $\tilde{G}_i$, $1\leq i\leq r$ as follows:

$$\tilde{G}_i=\left\{\begin{array}{ll} \GL(V_i) & \mbox{if $G_i=\SL(V_i)$ and $p|\dim V_i$,} \\
G_i & \mbox{otherwise.} \end{array}\right.$$
Set $\tilde{\mathfrak g}_i=\Lie(\tilde{G}_i)$ ($1\leq i\leq r$), $\tilde{G}=\tilde{G}_1\times\ldots\times\tilde{G}_r$ and $\tilde{\mathfrak g}=\Lie(\tilde{G})$.
Consider ${\mathfrak g}'$ as a Lie subalgebra of both ${\mathfrak g}$ and $\tilde{\mathfrak g}$.

Here we prove a generalization of a reduction theorem of Gordon and Premet \cite[6.2]{gandp}, extended to the case $m=2$ by the author in \cite{invs}.
This can be proved in a similar way to the $m=2$ case, and therefore we refer the reader to \cite[Thm. 3.1]{invs} for some details.
An important corollary is that the non-degenerate form $\kappa$ in (C) may be chosen to be $\theta$-equivariant.

\begin{proposition}\label{reduction}
There exists a torus $T_0$, an automorphism $\hat\theta$ of $\hat{G}=\tilde{G}\times T_0$ and a restricted Lie algebra embedding $\phi:{\mathfrak g}\rightarrow\hat{\mathfrak g}=\Lie(\hat{G})$ such that:

(i) $\hat\theta$ has order $m$,

(ii) $\phi(d\theta(x))=d\hat\theta(\phi(x))$ for all $x\in{\mathfrak g}$,

(iii) There is a $d\hat\theta$-stable toral algebra ${\mathfrak t}_1$ such that $\hat{\mathfrak g}=\phi({\mathfrak g})\oplus{\mathfrak t}_1$ (Lie algebra direct sum),

(iv) $\hat\theta$ stabilizes $\tilde{G}$ and $T_0$, and $\hat\theta(\tilde{G}_i)=\tilde{G}_j$ whenever $\theta(G_i)=G_j$.

\end{proposition}

\begin{proof}
The existence of a toral algebra ${\mathfrak s}_0$ and an injective restricted Lie algebra homomorphism $\eta:{\mathfrak g}\rightarrow\tilde{\mathfrak g}\oplus{\mathfrak s}_0$ such that $\eta({\mathfrak g}_i)={\mathfrak g}_i\subseteq\tilde{\mathfrak g}_i$ was proved by Premet \cite[Lemma 4.1]{premcomp}.
(This holds without the assumption (C).)
Moreover, by Gordon-Premet \cite[6.2]{gandp} there exists a toral algebra ${\mathfrak t}_1\subset\hat{\mathfrak g}$ such that $\hat{\mathfrak g}=\eta({\mathfrak g})\oplus{\mathfrak t}_1$.
Identify ${\mathfrak g}$ with its image $\eta({\mathfrak g})$, and define a restricted Lie algebra automorphism $\phi$ of $\hat{\mathfrak g}$ by $\phi(x)=d\theta(x)$ ($x\in{\mathfrak g}$), $\phi(t)=t$ ($t\in{\mathfrak t}_1$) and linear extension to all of $\hat{\mathfrak g}$.
The essential idea is to find $\phi$-stable subalgebras $\overline{\mathfrak g}$ and ${\mathfrak t}_0$ of ${\mathfrak g}$ such that $\overline{\mathfrak g}$ contains ${\mathfrak g}$ and is isomorphic to $\tilde{\mathfrak g}$, ${\mathfrak t}_0$ is a toral algebra and $\hat{\mathfrak g}=\overline{\mathfrak g}\oplus{\mathfrak t}_0$.

Let $(B,T)$ be a fundamental pair in $G$ for $\theta$, let ${\mathfrak h}=\Lie(T)$, $T'=T\cap G^{(1)}$, ${\mathfrak h}'=\Lie(T')={\mathfrak h}\cap{\mathfrak g}'$, $T_i=T\cap G_i$, let $\tilde{T}_i$ (resp. $\tilde{T},\hat{T}$) be the unique maximal torus of $\tilde{G}_i$ (resp. $\tilde{G}$, $\hat{G}$) containing $T_i$ (resp. $T'$) and let ${\mathfrak h}_i={\mathfrak h}\cap{\mathfrak g}_i=\Lie(T_i)$, $\tilde{\mathfrak h}_i=\Lie(\tilde{T_i})$, $\tilde{\mathfrak h}=\Lie(\tilde{T})$, $\hat{\mathfrak h}=\Lie(\hat{T})$.
Let $\Phi=\Phi(G,T)$ be the roots of $G$ relative to $T$, let $\Phi_i=\Phi(G_i,T\cap G_i)\subset\Phi$ and let $\Delta$ (resp. $\Delta_i$) be the basis of $\Phi$ (resp. $\Phi_i$) corresponding to $B$ (resp. $B\cap G_i$).
Clearly $\Delta=\cup_{i=1}^r\Delta_i$, and any element of $\Phi_i$ can be considered as an element of $X(\hat{T})$ (hence also $X(T)$, $X(\tilde{T})$, $X(T')$).

We first construct the $\phi$-stable toral algebra ${\mathfrak t}_0$.
Let ${\mathfrak z}={\mathfrak z}({\mathfrak g})$, $\tilde{\mathfrak z}={\mathfrak z}(\tilde{\mathfrak g})$, ${\mathfrak z}_i={\mathfrak z}({\mathfrak g}_i)$, $\hat{\mathfrak z}={\mathfrak z}(\hat{\mathfrak g})$.
Clearly $\hat{\mathfrak z}={\mathfrak z}\oplus{\mathfrak s}_0=\tilde{\mathfrak z}\oplus{\mathfrak t}_1$ and $\tilde{\mathfrak z}=\sum{\mathfrak z}_i$, thus $\tilde{\mathfrak z}\subseteq\hat{\mathfrak z}$ are $\phi$-stable toral algebras.
It follows by Maschke's theorem that there is a $\phi$-stable toral algebra ${\mathfrak t}_0^{tor}$ such that $\hat{\mathfrak z}^{tor}={\mathfrak t}_0^{tor}\oplus\tilde{\mathfrak z}^{tor}$.
Let ${\mathfrak t}_0$ be the (toral) subalgebra of $\hat{\mathfrak h}$ generated by ${\mathfrak t}_0^{tor}$.
The problem at this point (which does not arise for $m=2$) is that a toral algebra endowed with an arbitrary (restricted Lie algebra) automorphism cannot in general be described as the Lie algebra of a torus with algebraic automorphism.
Let $Z=Z(G)^\circ$ and let $Y(Z)$ be the group of cocharacters of $Z$.
The action of $\theta$ on $Z$ induces a ${\mathbb Z}$-module automorphism $\theta_Z:Y(Z)\rightarrow Y(Z)$.
Let $c(t)\in{\mathbb F}_p[t]$ be the reduction modulo $p$ of the characteristic polynomial of $\theta_Z$ and let $\tilde{c}(t)\in{\mathbb F}_p[t]$ be the characteristic polynomial of $\phi|_{{\tilde{\mathfrak z}^{tor}}}$.
Then (since ${\mathfrak z}=\Lie(Z)$ \cite[4.1]{comms}) the characteristic polynomial of $\phi|_{{\mathfrak t}_0^{tor}}$ is $(t-1)^{{\dim{\mathfrak t}_1}}c(t)/\tilde{c}(t)$.
Define a restricted Lie algebra automorphism of (the Lie algebra direct sum) $\hat{\mathfrak g}\oplus\tilde{\mathfrak z}$ by $(x,y)\mapsto (\phi(x),d\theta(y))$.
Clearly $\hat{\mathfrak g}\oplus\tilde{\mathfrak z}={\mathfrak g}\oplus{\mathfrak t}_1\oplus\tilde{\mathfrak z}=\tilde{\mathfrak g}\oplus{\mathfrak t}_0\oplus\tilde{\mathfrak z}$.
Hence, replacing $\hat{\mathfrak g}$ by $\hat{\mathfrak g}\oplus\tilde{\mathfrak z}$ and ${\mathfrak t}_0$ by ${\mathfrak t}_0\oplus\tilde{\mathfrak z}$, we may assume that $\phi|_{{{\mathfrak t}_0}^{tor}}$ has characteristic polynomial $(t-1)^{\dim{\mathfrak t}_1}c(t)$.
It is now clear that there exists a torus $T_0$ and a rational automorphism $\psi$ of $T_0$ such that $\Lie(T_0)={\mathfrak t}_0$ and $d\psi=\phi|_{{\mathfrak t}_0}$.

Denote by $\sigma$ the permutation of the set $\{Ê1,2,\ldots ,r\}$ such that $\sigma({\mathfrak g}_i)={\mathfrak g}_{\sigma(i)}$.
Next, we construct subalgebras $\overline{\mathfrak g}_i$ containing ${\mathfrak g}_i$ such that $\overline{\mathfrak g}_i\cong\tilde{\mathfrak g}_i$, $\phi(\overline{\mathfrak g}_i)=\overline{\mathfrak g}_{\sigma(i)}$ and $\sum_{i=1}^r\overline{\mathfrak g}_i\oplus{\mathfrak s}_0=\hat{\mathfrak g}$.
There is nothing to do unless $G_i=\SL(V_i)$, where $p|\dim V_i$.
Hence assume that $\tilde{G}_i=\GL(V_i)$ and $p|\dim V_i$.
After renumbering we can clearly assume that $i=1$, and that for some $l|m$, $\theta(G_i)=G_{i+1}$ ($1\leq i<l$) and $\theta(G_l)=G_1$.
By the argument in \cite[Step 2, Pf. of Thm. 3.1]{invs} it is straightforward to construct a $\phi^l$-stable toral subalgebra $\overline{\mathfrak h}_1$ of $\hat{\mathfrak h}$ which contains ${\mathfrak h}_1$ and such that $\cap_{\alpha\in\Delta\setminus\Delta_1}\ker d\alpha=\overline{\mathfrak h}_1\oplus\sum_{i>2}{\mathfrak z}_i\oplus{\mathfrak t}_0$.
Moreover, by linear independence of the differentials $d\alpha$, $\alpha\in\Delta$ \cite[4.2]{comms}, the restricted Lie subalgebra $\overline{\mathfrak g}_1=\overline{\mathfrak h}_1+{\mathfrak g}_1$ of $\hat{\mathfrak g}$ is isomorphic to $\tilde{\mathfrak g}_1$ (as a restricted Lie algebra).
It suffices now to take $\overline{\mathfrak g}_i=\phi^i(\overline{\mathfrak g}_1)$, $2\leq i\leq l$.
This construction (applied to all minimal $\phi$-stable summands in the expression ${\mathfrak g}'=\oplus{\mathfrak g}_i$) provides the required decomposition $\hat{\mathfrak g}=\oplus_{i=1}^r\overline{\mathfrak g}_i\oplus{\mathfrak t}_0$ which is preserved by the action of $\phi$.
Replacing $\tilde{\mathfrak g}$ by $\overline{\mathfrak g}$ in the obvious way (see the argument at the end of \cite[Pf. of Thm. 3.1]{invs}) we may assume that $\phi(\tilde{\mathfrak g}_i)=\tilde{\mathfrak g}_{\sigma(i)}$.
We claim that the restriction $\phi|_{\tilde{\mathfrak g}}$ is the differential of a rational automorphism $\tilde\theta$ of $\tilde{G}$.
Indeed, we need clearly only prove this for the restriction of $\phi$ to the sum of $\tilde{\mathfrak g}_i$ satisfying $\tilde{\mathfrak g}_i\neq{\mathfrak g}_i$, and hence we may assume as above that $\tilde{G}_1=\GL(V_1)$, that $\phi({\mathfrak g}_i)={\mathfrak g}_{i+1}$ ($1\leq i< l$) and $\phi({\mathfrak g}_l)={\mathfrak g}_1$.
Let $m'$ be the order of $\phi^l|_{\tilde{\mathfrak g}_1}$.
By Lemma \ref{GLnautos} there exists a unique automorphism $\psi_1$ of $\GL(V_1)$ of order $m'$ such that $d\psi_1=\phi^l|_{\tilde{\mathfrak g}_1}$.
Hence let $\tilde\theta$ act on $\tilde{G_1}\times \ldots \times \tilde{G}_l$ via $(g_1,\ldots ,g_l)\mapsto (\psi_1(g_l),g_1,\ldots ,g_{l-1})$.
Extending $\tilde{\theta}$ to $\tilde{G}\times T_0$ by $(g,t)\mapsto(\tilde\theta(g),\psi(t))$ gives the required automorphism of $\hat{G}$.
\end{proof}

As a consequence, we have:

\begin{corollary}
There exists a non-degenerate symmetric bilinear form $\kappa:{\mathfrak g}\times{\mathfrak g}\rightarrow k$.
\end{corollary}

\begin{proof}
The argument of \cite[Cor. 3.2]{invs} applies verbatim.
\end{proof}

\section{The Little Weyl Group}

It was proved in \cite{vin} for the case $k={\mathbb C}$ that the `little Weyl group' $W_{\mathfrak c}=N_{G(0)}({\mathfrak c})/Z_{G(0)}({\mathfrak c})\hookrightarrow\GL({\mathfrak c})$ is generated by pseudoreflections.
(Recall that an element $g\in\GL(V)$ of finite order is a pseudoreflection if the space of fixed points $V^g$ is of codimension 1 in $V$.)
It follows that the ring of invariants $k[{\mathfrak c}]^{W_{\mathfrak c}}$ is a polynomial ring.
In the modular case $W_{\mathfrak c}$ may have order divisible by the characteristic of the ground field.
On the other hand, we show in this section that it is sufficiently `nice' for the invariants to be polynomial, at least under the assumptions of the standard hypotheses.
Prop. \ref{reduction} essentially reduces us to the case that $G$ is almost simple, not of type $A_{ip-1}$, or that $G$ is isomorphic to $\GL(V)$ for a vector space $V$ of dimension divisible by $p$.
For $G$ of classical type Vinberg \cite{vin} has described the little Weyl group for all automorphisms.
One could use the same approach to verify that Vinberg's description holds in good characteristic.
(Most calculations are omitted in \cite{vin}.)
However, we provide a slightly different perspective here that also makes clear the precise relationship between the Weyl group of $G$ and $W_{\mathfrak c}$.
For $G$ of exceptional type we apply a result of Panyushev \cite{panorbits} and an inspection of orders of centralizers in the Weyl group (classified in \cite{carter}) to deduce the required result.
We assume $p>2$ from now on.

\begin{lemma}\label{treg}
Let ${\mathfrak c}$ be a Cartan subspace of ${\mathfrak g}(1)$, let $T_1$ be the unique minimal $\theta$-stable torus whose Lie algebra contains ${\mathfrak c}$ (Lemma \ref{thetasplit}), and let $T_m$ be a maximal torus of $(Z_G({\mathfrak c})^\theta)^\circ$.
Then $T_mT_1$ is regular in $G$.
Moreover, if $T=Z_G(T_1T_m)$ then $T_1$ and $T_m$ are the subtori of $T$ constructed before Lemma \ref{stabletori}.
\end{lemma}

\begin{proof}
Let $L=Z_G(T_1)=Z_G({\mathfrak c})$, a $\theta$-stable Levi subgroup of $G$.
Then $T_m$ is regular in $L$ by Lemma \ref{reg0}.
The final statement is clear.
\end{proof}

From now on ${\mathfrak c}$, $T_m$, $T_1$ will be as in Lemma \ref{treg} and $T$ will be the unique maximal torus of $G$ containing $T_m$ and $T_1$, unless otherwise stated.
Let $W_{\mathfrak c}:=N_{G(0)}({\mathfrak c})/Z_{G(0)}({\mathfrak c})$, let $G^\theta_Z=\{ g\in G\mid g^{-1}\theta(g)\in Z(G)\}$ and let $W_{\mathfrak c}^Z:=N_{G_Z^\theta}({\mathfrak c})/Z_{G_Z^\theta}({\mathfrak c})$.
Clearly both $W_{\mathfrak c}$ and $W^Z_{\mathfrak c}$ are invariant under isogeny.
In general $W_{\mathfrak c}\neq W_{\mathfrak c}^Z$: recall that $\theta$ is {\it saturated} if $W_{\mathfrak c}=W^Z_{\mathfrak c}$ (\cite[\S 5]{vin}).
Let $W=N_G(T)/T$.
Since $T$ is $\theta$-stable, $\theta$ acts on $W$.

\begin{lemma}\label{inclusion}
$W^Z_{\mathfrak c}$ (and hence $W_{\mathfrak c}$) embeds naturally as a subgroup of $W^\theta/Z_{W^\theta}({\mathfrak c})$.
\end{lemma}

\begin{proof}
Suppose $g\in G^\theta_Z$ normalizes ${\mathfrak c}$.
Then $g$ normalizes $L=Z_G({\mathfrak c})=Z_G(T_1)$ (by Lemma \ref{thetasplit}) and hence $g^{-1}T_mg$ is a maximal torus of $L(0)=(L^\theta)^\circ\subset Z_{G(0)}({\mathfrak c})$.
Therefore, after replacing $g$ by $gh$ for suitable $h\in L(0)$, we may assume that $g$ normalizes $T_m$, hence that $g$ normalizes $T$.
It follows that each element of $W^Z_{\mathfrak c}$ has a representative in $W$.
But such a representative must clearly be in $W^\theta$.
\end{proof}

In general the inclusion in Lemma \ref{inclusion} may be proper.
From now on let $W_1=W^\theta/Z_{W^\theta}({\mathfrak c})$.
It is easy to see that $W_1$ normalizes ${\mathfrak c}$.
As before, $r$ will denote the rank and $m$ the order of $\theta$.
Let $T_i$, $i|d$ be the subtori of $T$ defined in Lemma \ref{stabletori}.

\begin{lemma}\label{criter}
Let $T'_m=\prod_{i\neq 0}T_i=\{ t^{-1}\theta(t)\mid t\in T\}$.

(a) Suppose $\{ t\in T_m\mid t^m=1\}\subset T'_m$.
If $G^\theta=G(0)$ (if, for example, $G$ is semisimple and simply-connected), then $W_{\mathfrak c}=W_1$.

(b) Suppose $\{ t\in T_m\mid t^m=1\}\subset T'_m Z(G)$.
Then $W_{\mathfrak c}^Z=W_1$.
\end{lemma}

\begin{proof}
Let ${\cal T}=\{ t\in T_m\mid t^m=1\}$.
We claim that $\{ t\in T\mid t\theta(t)\ldots \theta^{m-1}(t)=1\}={\cal T}\cdot T'_m$.
Indeed, it is clear from Lemma \ref{stabletori} that $t\theta(t)\ldots \theta^{m-1}(t)=1$ for any $t\in T'_m$.
Since $T=T_m\cdot T'_m$, the equality follows.
Thus let $w=n_wT\in W^\theta$.
Then $x=n_w^{-1}\theta(n_w)\in T$.
But clearly $x\theta(x)\ldots\theta^{m-1}(x)=1$, and hence $x\in{\cal T}\cdot T'_m$.
If ${\cal T}\subset T'_m$, then $x$ is contained in the image of the map $T\rightarrow T$, $t\mapsto t^{-1}\theta(t)$.
Thus $x=t\theta(t^{-1})$ for some $t\in T$ and hence $n_wt\in G^\theta$.
If $G^\theta=G(0)$, it follows that $W_{\mathfrak c}=W_1$.
(If $G$ is semisimple and simply-connected then $G^\theta=G(0)$ by \cite[8.1]{steinberg}.)
Similarly, if ${\cal T}\subset T'_m Z(G)$ then $x=t\theta(t^{-1})z$ for some $t\in T$, $z\in Z(G)$ and thus $n_wt\in G_Z^\theta$.
This proves (b).
\end{proof}

With the aid of Lemma \ref{criter}, we now determine the little Weyl group in the case where $G$ is one of the classical groups $\SL(n,k)$, $\SO(n,k)$, $\Sp(2n,k)$.
Following \cite{vin}, we call $({\mathfrak g},d\theta)$ associated to such a group a {\it classical graded Lie algebra}.
One apparent problem here is that $\SO(n,k)$ is not simply-connected.
However, the universal covering $\Spin(n,k)\rightarrow\SO(n,k)$ is separable and hence any classical graded Lie algebra is the Lie algebra of a group (with automorphism) satisfying the standard hypotheses.
On the other hand, all automorphisms of $\Spin(n,k)$ give rise to automorphisms of $\SO(n,k)$ unless $n=8$.
This is obvious if $n$ is odd since then $\SO(n,k)$ is just the quotient of $\Spin(n,k)$ by its centre.
Let $\hat{T}$ be a maximal torus of $\Spin(2n,k)$, let $\Phi(\Spin(2n,k),\hat{T})$ be identified with the root system $\Phi$ of $\SO(2n,k)$, let $\Delta=\{\alpha_1,\ldots ,\alpha_n\}$ be a basis of $\Phi$ (numbered in the standard way) and let $\alpha_i^\vee:k^\times\rightarrow\hat{T}$ be the corresponding coroots.
Let $z_0,z_1\in\Spin(2n,k)$: $$z_0=\alpha^\vee_{n-1}(-1)\alpha_n^\vee(-1),$$  $$z_1=\left\{\begin{array}{ll} \alpha_1^\vee(-1)\alpha_3^\vee(-1)\ldots\alpha_{n-1}^\vee(-1) & \mbox{if $n$ is even,} \\ \alpha_1^\vee(-1)\alpha_3^\vee(-1)\ldots\alpha_{n-1}^\vee(i)\alpha_n^\vee(-i) & \mbox{if $n$ is odd.}\end{array}\right.$$
It is well known (and easy to show) that: $$Z(\Spin(2n,k))=\left\{\begin{array}{ll} \{ 1,z_0,z_1,z_0z_1\}\cong({\mathbb Z}/2{\mathbb Z})^2 & \mbox{if $n$ is even,} \\ \{ 1,z_1,z_1^2=z_0,z_1^{3}\}\cong{\mathbb Z}/4{\mathbb Z} & \mbox{if $n$ is odd,}\end{array}\right.$$ and the kernel of the covering morphism $\Spin(2n,k)\rightarrow\SO(2n,k)$ is generated by $z_0$.

\begin{lemma}
(a) If $n>4$ then any rational automorphism $\theta$ of $\Spin(2n,k)$ satisfies $\theta(z_0)=z_0$.
Hence $\Aut\Spin(2n,k)\cong\Aut\SO(2n,k)\cong\Aut(\SO(2n,k)/\{\pm I\})\cong{\rm O}(2n,k)/\{\pm I\}$.

(b) $\Aut\Spin(8,k)/\Int\Spin(8,k)$ is isomorphic to the symmetric group $S_3$.
If $\theta$ is a rational automorphism of $\Spin(8,k)$ then either $\theta^2$ is inner, in which case some $\Aut\Spin(8,k)$-conjugate of $\theta$ preserves $z_0$, or $\theta^3$ is inner.
\end{lemma}

\begin{proof}
We need only check (a) for outer automorphisms, hence for a particular choice of outer automorphism.
But there exists an outer automorphism $\theta$ which satisfies $\theta(\alpha_i^\vee(t))=\alpha_i^\vee(t)$ ($1\leq i\leq n-2$), $\theta(\alpha_{n-1}^\vee(t))=\alpha_n^\vee(t)$ and $\theta(\alpha_n^\vee(t))=\alpha_{n-1}^\vee(t)$.
Hence (a) follows.
Finally, (b) follows immediately from well-known properties of automorphisms of reductive groups, see for example \cite[27.4]{hum}.
\end{proof}

For $m,r\in{\mathbb N}$ and $q$ dividing $m$ let $G(m,q,r)$ denote the subgroup of $\GL(r,k)$ consisting of all monomial matrices with entries $x_i$ satisfying $x_i^m=1$, $(\prod_{i=1}^rx_i)^{m/q}=1$.
Our description below of $W_{\mathfrak c},W_{\mathfrak c}^Z,W_1$ using this notation refers to the action on ${\mathfrak c}$.

\begin{lemma}\label{sln}
Let $G=\SL(n,k)$, $p\nmid n$ or $G=\GL(n,k)$ and let $\theta$ be inner.
Then $W_{\mathfrak c}=W_{\mathfrak c}^Z=W_1=G(m,1,r)$.
\end{lemma}

\begin{proof}
Since $\theta$ is inner and stabilizes $T$, it equals $\Int n_w$ for some $n_w\in N_G(T)$.
Now since $T_m$ is maximal in $Z_G({\mathfrak c})^\theta$, we claim that $w=n_w T$ is a product of $r$ $m$-cycles.
Indeed, let $w=w_m\cdot w'_m$ be the decomposition of $w$, where $w_m$ is a product of $m$-cycles and $w'_m$ is a product of cycles of order less than $m$.
Let $L=Z_G({\mathfrak c})$.
Then it is easy to see that $L^{(1)}\cong\SL(n-rm,k)$, that $\Lie(L)$ is the span of ${\mathfrak t}$ and all root subspaces ${\mathfrak g}_\alpha$ with $w_m(\alpha)=\alpha$, that $\theta|_{L^{(1)}}$ is an inner automorphism, and hence by our assumption on $T_m$ that $w'_m=1$.
Choosing a suitable $N_G(T)$-conjugate of $n_w$, we may assume that $w$ has the form $$\begin{pmatrix} 1 & \ldots & m \end{pmatrix} \begin{pmatrix} m+1 & \ldots & 2m \end{pmatrix}\ldots \begin{pmatrix} (r-1)m+1 & \ldots & rm \end{pmatrix}$$
Hence it is clear that we can choose a basis $\{ c_1,\ldots,c_r\}$ for ${\mathfrak c}$, where $c_i$ is the diagonal matrix with $j$-th diagonal entry: $\left\{\begin{array}{ll} \zeta^{-j} & \mbox{if $(i-1)m<j\leq im$,} \\ 0 & \mbox{otherwise.}\end{array}\right.$

With this description it is immediate that $W_1=G(m,1,r)$.
Let $L=Z_G({\mathfrak c})$ and let $S$ be a maximal torus of $L^{(1)}$.
Since any element of $W_1$ has a representative in $Z_G(S)$ (hence in $Z_G(S)^{(1)}\cong\SL(rm,k)$), we may assume that $n=rm$.
Now it is clear that any element of $T_m$ has the form $$\begin{pmatrix}Êt_1 I_m & \cdots & 0 \\ \vdots & \ddots & \vdots \\ 0 & \cdots & t_rI_m \end{pmatrix}$$ (where $I_m$ is the $m\times m$ identity matrix and $t_1,\ldots ,t_r\in k^\times$), and that such an element is in ${\cal T}$ if and only if each $t_i$ is a power of $\zeta$.
(Recall that $\zeta$ is a fixed primitive $m$-th root of unity.)
We therefore prove that ${\cal T}\subset T'_m$ in the case $r=1$; this will make it clear that the inclusion holds for arbitrary $r$.
If $m$ is odd, then the matrix $s=\diag (\zeta^{m-1},\zeta^{m-2},\ldots ,1)$ is of determinant 1 and satisfies $s^{-1}\theta(s)=\zeta I_m$.
Hence ${\cal T}\subset \{ t^{-1}\theta(t)\mid t\in T\}=T'_m$.
If $m$ is even, then let $\xi$ be a square-root of $\zeta$.
Then $\xi s$ is of determinant 1 and $\xi s\theta(\xi s)^{-1}=\zeta I_m$.
Thus ${\cal T}\subset T'_m$ in this case as well.
Applying Lemma \ref{criter}, this completes the proof.
\end{proof}

The above result corresponds to the `First case' in Vinberg's classification, \cite[\S 7]{vin}.

\begin{rk}\label{ark}
We recall that the automorphism $\theta$ is {\bf $S$-regular} if ${\mathfrak g}(1)$ contains a regular semisimple element of ${\mathfrak g}$.
It can easily be seen from the proof of Lemma \ref{sln} that here there is a $\theta$-stable Levi subgroup of $G$ such that ${\mathfrak c}$ is contained in the Lie algebra of its derived subgroup $L$, $N_{L(0)}({\mathfrak c})/Z_{L(0)}({\mathfrak c})\cong W_{\mathfrak c}$ and the restriction of $\theta$ to $L$ is $S$-regular.
We have $L\cong \SL(rm,k)$.
In fact, $L=H^{(1)}$, where $H$ is a minimal Levi subgroup of $G$ whose Lie algebra contains $T_1$.
We will see in Sect. 5 that $\theta|_L$ is in fact $N$-regular, that is, ${\mathfrak l}(1)$ contains a regular nilpotent element of ${\mathfrak l}$.
\end{rk}

For the remaining classical cases, we require a little preparation.
Let $J_n$ denote the $n\times n$ matrix with 1 on the antidiagonal and 0 elsewhere and let $\gamma:\GL(n,k)\rightarrow\GL(n,k)$, $g\mapsto {^t}g^{-1}$.
(By abuse of notation we will use $\gamma$ to denote this automorphism for arbitrary $n$.)
In our setting, ${\rm O}(n,k)$ is the group of $n\times n$ matrices which are stable under $\Int J_n\circ\gamma$, $\SO(n,k)$ is the intersection of ${\rm O}(n,k)$ with $\SL(n,k)$ and $\Sp(2n,k)$ is the subgroup of fixed points in $\SL(2n,k)$ under the automorphism $\Int \begin{pmatrix}Ê0 & J_n \\ -J_n & 0 \end{pmatrix}\circ\gamma$.
Until further notice $G$ will be one of $\SO(2n,k)$, $\SO(2n+1,k)$, $\Sp(2n,k)$.
We will choose $T$ to be the maximal torus of diagonal matrices in $G$:
$$T=\left\{\begin{pmatrix} t_1 & \cdots & 0 \\ \vdots & \ddots & \vdots \\ 0 & \cdots & t_1^{-1}\end{pmatrix}\mid t_1,\ldots ,t_n\in k^\times\right\}$$
For the purposes of describing the action of the Weyl group, we identify $T$ with $(k^\times)^n$ via the isomorphism $(k^\times)^n\rightarrow T$, $t=(t_1,\ldots ,t_n)\mapsto\diag(t_1,\ldots ,t_1^{-1})$.
Let $\overline{G}={\rm O}(2n,k)$ if $G=\SO(2n,k)$, and let $\overline{G}=G$ otherwise.
Let $\overline{W}=N_{\overline{G}}(T)/T$.
Then $\overline{W}\cong S_n\ltimes(\mu_2)^n\cong G(2,1,n)$, where $\mu_2$ is the multiplicative group $\{\pm 1\}$.
(If $G=\SO(2n,k)$ then $W\cong G(2,2,n)$.)
Specifically, elements of $S_n$ act as permutations $(t_1,t_2,\ldots ,t_n)\mapsto (t_{\sigma(1)},\ldots ,t_{\sigma(n)})$, and $(\epsilon_1,\ldots ,\epsilon_n)\in{\bf \mu}_2^n$ sends $(t_1,\ldots ,t_n)$ to $(t_1^{\epsilon_1},\ldots ,t_n^{\epsilon_n})$.
There is a classification of the conjugacy classes in $\overline{W}$ by {\it signed cycle types}.
That is, if $w\in \overline{W}$ is conjugate to $\sigma=\begin{pmatrix} 1 & \ldots & l \end{pmatrix}\in S_n$ (resp. to $((-1,1,1,\ldots ,1),\sigma)\in({\bf \mu}_2)^n\ltimes S_n$) then we say that $w$ is a positive (resp. negative) $l$-cycle.
A positive (resp. negative) $l$-cycle is of order $l$ (resp. $2l$).
Extending in the obvious way to products of disjoint cycles one can there associated a (unique) signed permutation type to each $w\in\overline{W}$.
This correspondence is one-to-one between conjugacy classes in $\overline{W}$ and signed cycle types $1^{a_1}\overline{1}^{b_1}\ldots l^{a_l}\overline{l}^{b_l}$ with $\sum_{i=1}^l i(a_i+b_i)=n$ (\cite[Prop. 24]{carter}).
(Here $i$ denotes a positive $i$-cycle and $\overline{i}$ denotes a negative $i$-cycle.)
Let $J'=\begin{pmatrix} I_{n-1} & 0 & 0 \\ 0 & J_2 & 0 \\ 0 & 0 & I_{n-1} \end{pmatrix}\in{\rm O}(2n,k)$ (where $I_{n-1}$ denotes the $(n-1)\times (n-1)$ identity matrix).

\begin{lemma}\label{so}
Any semisimple element of ${\rm O}(2n,k)$ is conjugate to an element of $T\cup J'T$.
\end{lemma}

\begin{proof}
Since any semisimple element of $G=\SO(2n,k)$ is conjugate to an element of $T$, it will clearly suffice to show that any semisimple element of $J'G$ is conjugate to an element of $J'T$.
But if $J'g$ is semisimple then $\Int (J'g)$ stabilizes a maximal torus of $G$ and a Borel subgroup containing it.
Let $B$ be the intersection of $G$ with the group of upper-triangular $2n\times 2n$ matrices, a Borel subgroup which contains $T$.
Thus, after conjugating by a suitable element of $G$, we may assume that $J'g\in N_{{\rm O}(2n,k)}(T)$ and that $\Int J'g$ normalizes $B$.
But the result is now clear, since $N_B(T)=T$ and $\overline{W}/W$ is of order 2.
\end{proof}

\begin{lemma}\label{cycles}
Let $G$ be one of $\SO(2n+1,k)$, $\SO(2n,k)$ or $\Sp(2n,k)$ and let $m$ be even.
Then $\theta=\Int n_w$ for $n_w\in N_{\overline{G}}(T)$ where $w=n_w T\in\overline{W}$ is either a product of $r$ positive $m$-cycles, a product of $r$ negative $(m/2)$-cycles, or possibly a product of $r$ negative $(m/2)$-cycles and one negative $1$-cycle if $m>2$ and $G=\SO(2n,k)$.
If $G=\SO(2n+1,k)$ then $w$ is a product of $r$ negative $(m/2)$-cycles.

If $w$ is a product of positive $m$-cycles then $$n_w^m=\left\{\begin{array}{ll} I & \mbox{if $G=\Sp(2n,k)$,} \\ -I & \mbox{if $G=\SO(2n,k)$.}\end{array}\right.$$

If $w$ is a product of negative cycles then $$n_w^m=\left\{\begin{array}{ll} -I & \mbox{if $G=\Sp(2n,k)$,} \\ I & \mbox{otherwise.}\end{array}\right.$$
\end{lemma}

\begin{proof}
Suppose that $G=\SO(m,k)$ or $G=\Sp(m,k)$ and that $w$ is a single negative $(m/2)$-cycle.
Then $n_w^m=\pm I$ and the characteristic polynomial of $n_w$ is, correspondingly, $T^m\mp I$.
If $G=\SO(m,k)$ (resp. $G=\Sp(m,k)$) then $n_w\not\in G$ (resp. $n_w\in G$) and hence $\det n_w=-1$ (resp. $\det n_w=1$), from which it follows that $n_w^m=I$ (resp. $n_w^m=-I$).
Suppose now that $G=\SO(2m,k)$ or $G=\Sp(2m,k)$ and that $w$ is a single positive $m$-cycle.
Then $n_w^m=\pm I$.
Note that $n_w$ is a monomial matrix in $\SL(2m,k)$ which corresponds to a product of two $m$-cycles.
Let $\xi$ be a square-root of $\zeta$.
Suppose that $n_w^m=I$ (resp. $n_w^m=-I$) and $G=\SO(2m,k)$ (resp. $G=\Sp(2m,k)$).
Then $n_w$ has eigenvalues $\zeta^i$ (resp. $\xi^{2i+1}$), $0\leq i<m$ and each eigenvalue is of multiplicity two.
But then $n_w$ is $G$-conjugate to an element of $N_G(T)$ which acts as a product of two negative $m/2$-cycles on $T$, by the above.
This contradicts maximality of ${\mathfrak c}$, and therefore $n_w^m=-I$ (resp. $n_w^m=I$) if $G=\SO(2m,k)$ (resp. $G=\Sp(2m,k)$).
On the other hand, it is easy to check that if $G=\SO(2m,k)$ (resp. $G=\Sp(2m,k)$) and $g\in G$ is conjugate to $\diag (\xi^{2m-1},\xi^{2m-3},\ldots ,\xi)$ (resp. $\diag (\zeta^{-1},\ldots ,\zeta,1,1,\zeta^{-1},\ldots ,\zeta)$ then $\Int g$ is a rank one automorphism of $G$.

We have therefore proved that there is a unique conjugacy class of automorphism of order $m$ of $\SO(m,k)$ (resp. $\Sp(m,k)$, $\SO(2m,k)$. $\Sp(2m,k)$) which acts as a negative $m/2$-cycle (resp. negative $m/2$-cycle, positive $m$-cycle, positive $m$-cycle).
Let us therefore consider the general case of the lemma.
Let $w=w^+_mw^-_mw'_m$, where $w^+_m$ is a product of $r_1$ positive $m$-cycles, $w^-_m$ is a product of $r_2=r-r_1$ negative $(m/2)$-cycles, and $w'_m$ is a product of signed cycles of order less than $m$.
(Hence $w'_m=1$ if $m=2$.)
Let $w^+_m=w_1\ldots w_{r_1}$ where the $w_i$ are disjoint positive $m$-cycles and let $w^-_m=w_{r_1+1}\ldots w_{r}$, where the $w_i$ are disjoint negative $(m/2)$-cycles.
Let $\{ c_1,\ldots ,c_r\}$ be a basis for ${\mathfrak c}$ such that $w_i(c_j)=\zeta^{\delta_{ij}}c_j$.
It is easy to see that there exist $\theta$-stable subgroups $L_1,\ldots ,L_r$ of $\overline{G}$ such that $c_i\in \Lie(L_i)$ and $$L_i\cong \left\{\begin{array}{ll} {\rm O}(m,k) & \mbox{if $G=\SO(n,k)$ and $1\leq i\leq r_1$,} \\ {\rm O}(2m,k) & \mbox{if $G=\SO(n,k)$ and $r_1+1\leq i\leq r$,} \\ \Sp(m,k) & \mbox{if $G=\Sp(2n,k)$ and $1\leq i\leq r_1$,} \\ \Sp(2m,k) & \mbox{if $G=\Sp(2n,k)$ and $r_1+1\leq i\leq r$.}Ê\end{array}\right.$$
Then $\theta|_{L_i}=\Int x_i$ for some $x_i\in L_i$ and $x_i^m=\pm I$ according to the criteria given in the paragraph above.
Thus $n_w=x_1 z$, where $z\in Z_{\overline G}(L_1)$.
But then if $r_1>0$, $n_w^m=x_1^m z^m$, and therefore $n_w^m=-I$ (resp. $I$) if $G$ is of orthogonal (resp. symplectic) type.
Similarly, if $r_2>0$ then $n_w=x_{r}z$ for $z\in Z_G(L_r)$) and therefore $n_w^m=I$ (resp. $-I$) if $G$ is of orthogonal (resp. symplectic) type.
It follows that either $r_1=0$ or $r_2=0$, and if $G=\SO(2n+1,k)$ then $w$ is a product of negative cycles.
Moreover, $n_w\prod_1^r x_i^{-1}\in Z_{\overline G}({\mathfrak c})$ and represents $w'_m$ as an element of $\overline W$.
Thus $w'_m$ is trivial if $G=\Sp(2n,k)$ or $G=\SO(2n+1,k)$, and by Lemma \ref{so} $w'_m$ is either trivial or a single negative 2-cycle if $G=\SO(2n,k)$.
On the other hand, if $w'_m$ is a negative 2-cycle then clearly $n_w^m=I$ and thus $w$ is a product of negative cycles.
\end{proof}

\begin{lemma}\label{typec}
Let $G=\Sp(2n,k)$.

(a) If $m$ is odd then $W_{\mathfrak c}=W_{\mathfrak c}^Z=W_1=G(2m,1,r)$.

(b) If $m$ is even then $W_{\mathfrak c}=W_{\mathfrak c}^Z=W_1=G(m,1,r)$.
\end{lemma}

\begin{proof}
Since any automorphism of $G$ is inner, $\theta=\Ad n_w$ for some $n_w\in N_G(T)$.
If $m$ is odd, then $w=n_wT$ is a product of $r$ positive $m$-cycles by the argument in Lemma \ref{cycles}.
After conjugating by a suitable element of $N_G(T)$, we may assume that $$w=\begin{pmatrix} 1 & \ldots & m \end{pmatrix} \begin{pmatrix} m+1 & \ldots & 2m \end{pmatrix}\ldots \begin{pmatrix} (r-1)m+1 & \ldots & rm \end{pmatrix}$$
We can construct a basis $\{ c_1,\ldots ,c_r\}$ for ${\mathfrak c}$ in the same way as in the proof of Lemma \ref{sln}.
Then it is immediate that $W_1=G(2m,1,r)$.
Let $S$ be a $\theta$-stable maximal torus of $Z_G({\mathfrak c})^{(1)}$.
Then any element of $W_1$ has a representative in $Z_G(S)$, and hence in $Z_G(S)^{(1)}\cong\Sp(2rm)$.
Thus we may assume that $n=rm$.
As in the proof of Lemma \ref{sln}, it will clearly suffice to prove that ${\cal T}\subset T'_m$ in the case $r=1$.
Here $T_m$ consists of matrices of the form $\begin{pmatrix} tI_m & 0 \\ 0 & t^{-1}I_m \end{pmatrix}$.
But if $t=\diag(\zeta^{m-1},\zeta^{m-2},\ldots,\zeta,1,1,\zeta^{-1},\ldots,\zeta^2,\zeta)$ then $t^{-1}\theta(t)=\begin{pmatrix} \zeta I_m & 0 \\ 0 & \zeta^{-1}I_m \end{pmatrix}$.
By Lemma \ref{criter}, $W_{\mathfrak c}=W_1$.

For (b), Lemma \ref{cycles} shows that $w$ is either a product of $r$ positive $m$-cycles or a product of $r$ negative $m/2$-cycles.
It is easy to see by a similar argument to that used above that $W_1=G(m,1,r)$ in either case.
Let $S$ be a $\theta$-stable maximal torus of $Z_G({\mathfrak c})^{(1)}$: then $Z_G(S)^{(1)}$ is $\theta$-stable, isomorphic to $\Sp(2mr,k)$ (if $w$ is a product of positive $m$-cycles) or $\Sp(mr,k)$ (if $w$ is a product of negative $(m/2)$-cycles) and contains $T_1$ and a representative of each element of $W_1$.
Hence it will suffice to prove the equality $W_{\mathfrak c}=W_1$ in the case $n=mr$ ($w$ a product of positive $m$-cycles), $n=mr/2$ ($w$ a product of negative $(m/2)$-cycles).
For this we apply Lemma \ref{criter}.
If $w$ is a product of positive $m$-cycles then after conjugating by a suitable element of $N_G(T)$ we may assume that $w=\begin{pmatrix} 1 & \ldots & m \end{pmatrix} \ldots \begin{pmatrix}Ê(r-1)m+1 & \ldots & rm \end{pmatrix}$.
In these circumstances $T_m$ is the set of matrices of the form $$\begin{pmatrix}Êt_1I_m & \cdots & 0 \\ \vdots & \ddots & \vdots \\ 0 & \cdots & t_1^{-1}I_m \end{pmatrix}$$ where $t_1,\ldots ,t_r\in k^\times$.
Such an element is in ${\cal T}$ if and only if $t_i$ is a power of $\zeta$ for each $1\leq i\leq r$.
The inclusion ${\cal T}\subset T'_m$ can now be proved in exactly the same way as in Lemma \ref{sln}.
On the other hand, if $w$ is a product of negative $(m/2)$-cycles then it is easy to see that $T_m$ is trivial.
\end{proof}

In Vinberg's classification \cite[\S 7]{vin}, this is the `Third case': $m$ odd is Type III; $m$ even, $w$ a product of negative $(m/2)$-cycles is Type I, and $m$ even, $w$ a product of positive $m$-cycles is Type II.

\begin{rk}\label{rkc}
As for inner automorphisms in type $A$, it is easy to see from the proof of Lemma \ref{typec} that there is a $\theta$-stable Levi subgroup of $G$ such that ${\mathfrak c}$ is contained in the Lie algebra of its derived subgroup $L$, $N_{L(0)}({\mathfrak c})/Z_{L(0)}({\mathfrak c})\cong W_{\mathfrak c}$ and $\theta|_{L}$ is $S$-regular.
We have $L\cong\Sp(rm,k)$ if $m$ is even and $w$ is a product of negative $(m/2)$-cycles, $L\cong\SL(rm,k)$ if $m$ is even and $w$ is a product of positive $m$-cycles and $L\cong\Sp(2rm,k)$ if $m$ is odd.
(If $m$ is even and $w$ is a product of positive $m$-cycles then we can easily reduce to a subgroup isomorphic to $\Sp(2rm,k)$ whose Lie algebra contains ${\mathfrak c}$.
But now any element of ${\mathfrak c}$ is fixed by $\gamma$ (defined after Rk. \ref{ark}) and we can see by Lemma \ref{sln} that the little Weyl group for the restriction of $\theta$ to $G^\gamma\cong\GL(rm,k)$ is $G(m,1,r)$.
Here we restrict to $\SL(2rm,k)$ in order to ensure $N$-regularity.)
In common with type $A$, $L$ is the derived subgroup of a Levi subgroup $H$ of $G$, and $H$ is a minimal Levi subgroup whose derived subgroup contains $T_1$.
\end{rk}

\begin{lemma}\label{typeb}
Let $G$ be semisimple of type $B_n$.

(a) If $m$ is odd then $W_{\mathfrak c}=W_1=G(2m,1,r)$.

(b) If $m$ is even then $W_{\mathfrak c}=W_1=G(m,1,r)$.
\end{lemma}

\begin{proof}
Let $G=\SO(2n+1,k)$. While it is practical to work with $\SO(2n+1,k)$, things are slightly more diffficult than for $\Sp(2n,k)$ since centralizers are not in general connected.
If $m$ is odd, on the other hand, we claim that $G^\theta$ is connected.
Indeed, since all rational automorphisms of $G$ are inner, $\theta=\Ad n_w$ for some $n_w\in N_G(T)$.
Let $\pi:\hat{G}=\Spin(2n+1,k)\rightarrow G$ be the universal covering of $G$ and let $\hat{n}_{w}\in \hat{G}$ be such that $\pi(\hat{n}_{w})=n_w$.
Since the kernel of $\pi$ is just $Z(\hat{G})$, $G^\theta$ is disconnected if and only if there exists $x\in \hat{G}$ such that $x^{-1}\hat{n}_wx\hat{n}_w^{-1}$ is the non-identity element of $Z(\hat{G})$.
But $x^{-1}\hat{n}_wx\hat{n}_w^{-1}\in \{ h\in \hat{G}\mid h\theta(h)\ldots\theta^{m-1}(h)=1\}$ and hence $(x^{-1}\hat{n}_wx\hat{n}_w^{-1})^m=1$.
Thus $G^\theta=G(0)$ if $m$ is odd.
Since $w=n_wT$ contains $r$ positive $m$-cycles, it is straightforward to check that $Z_G({\mathfrak c})\cong (k^\times)^{rm}\times\SO(2(n-rm)+1,k)$.
Applying the argument in the proof of Lemma \ref{typec} and Lemma \ref{criter}(a), we deduce that $w$ is a product of $r$ positive $m$-cycles and that $W_1=W_{\mathfrak c}=G(2m,1,r)$.

Suppose therefore that $m$ is even.
By Lemma \ref{cycles}, $\theta=\Int n_w$, where $n_w^m=I$ and $w=n_wT$ is a product of $r$ negative $(m/2)$-cycles.
Using the same argument as in Lemma \ref{sln}, it follows that $W_1=G(m,1,r)$.
Let $S$ be a $\theta$-stable maximal torus of $Z_G({\mathfrak c})^{(1)}\cong \SO(2n+1-rm,k)$ and let $L=Z_G(S)^{(1)}$.
Then it is easy to see that $L\cong \SO(rm+1,k)$, that ${\mathfrak c}\subset\Lie(L)$ and that any element of $W_1$ has a representative in $L$.
Hence it will suffice to prove (b) under the assumption that $n=rm/2$.
But now $T_m$ is trivial.
Lifting $\theta$ (uniquely) to an automorphism of the universal covering $\hat{G}$ of $G$ (Lemma \ref{cover}), we can apply Lemma \ref{criter}.
\end{proof}

This is half of Vinberg's `Second case' (the other half being $\SO(2n,k)$): $m$ even, $w$ a product of negative $m/2$-cycles is Type I; $m$ odd is Type III.
(Type II, where $m$ is even and $w$ is a product of positive $m$-cycles does not occur by Lemma \ref{cycles}.)

\begin{rk}\label{rkb}
Once again, it is clear from the proof of Lemma \ref{typeb} that if $H$ is a minimal Levi subgroup of $G$ whose derived subgroup $L$ contains $T_1$ then $N_{L(0)}({\mathfrak c})/Z_{L(0)}({\mathfrak c})\cong W_{\mathfrak c}$ and $\theta|_{L}$ is $S$-regular.
We have $L\cong\SO(rm+1,k)$ if $m$ is even and $w$ is a product of negative $(m/2)$-cycles, and $L\cong\SO(2rm+1,k)$ if $m$ is odd.
\end{rk}

\begin{lemma}\label{typed}
Let $G=\SO(2n,k)$.
Then $\theta=\Int n_w$ for some $n_w\in N_{{\rm O}(2n,k)}(T)$.

(a) If $m$ is odd then:

(i) $W_1=\left\{\begin{array}{ll} G(2m,1,r) & \mbox{if $n>mr$,} \\ G(2m,2,r) & \mbox{if $n=mr$.}\end{array}\right.$

(ii) $W^Z_{\mathfrak c}=W_{\mathfrak c}=\left\{\begin{array}{ll} G(2m,1,r) & \mbox{if $n>mr$ and $Z_{{\rm O}(2n,k)}({\mathfrak c})^\theta\neq Z_G({\mathfrak c})^\theta$,} \\ G(2m,2,r) & \mbox{otherwise.}\end{array}\right.$

(b) If $m$ is even and $n_w^m=-I$, then $W_{\mathfrak c}=W_1=G(m,1,r)$.
If $m$ is even and $n_w^m=I$, then 

(i) $W_1=\left\{\begin{array}{ll} G(m,1,r) & \mbox{if $n>mr/2$,} \\ G(m,2,r) & \mbox{if $n=mr/2$.} \end{array}\right.$

(ii) $W_{\mathfrak c}^Z=W_{\mathfrak c}=\left\{\begin{array}{ll} G(m,1,r) & \mbox{if $n>mr/2$ and $Z_{{\rm O}(2n,k)}({\mathfrak c})^\theta\neq Z_G({\mathfrak c})^\theta$,} \\ G(m,2,r) & \mbox{otherwise.} \end{array}\right.$
\end{lemma}

\begin{proof}
Suppose $m$ is odd.
Then the kernel of the universal covering $\Spin(2n,k)\rightarrow G$ contains two elements, and hence we can apply the argument from Lemma \ref{typeb} to deduce that $G^\theta=G(0)$.
Since $Z(G)$ also has two elements, we can apply the same argument to the map $G\rightarrow G/Z(G)$ to deduce that $G_Z^\theta=G(0)$.
Since ${\rm O}(2n,k)/G$ has order 2, clearly $n_w\in G$.
Moreover, $w=n_wT$ contains $r$ positive $m$-cycles and hence it is straightforward to check that $Z_G({\mathfrak c})\cong (k^\times)^{mr}\times\SO(2(n-mr),k)$.
Since any odd-order automorphism of $\SO(2(n-mr),k)$ is inner, it follows by our choice of $T_m$ that $w$ is equal to a product of $r$ positive $m$-cycles.
After conjugating by a suitable element of $N_G(T)$, we may assume that $$w=\begin{pmatrix} 1 & \cdots & m\end{pmatrix}\ldots \begin{pmatrix} (m-1)r+1 & \cdots & mr \end{pmatrix}$$
Thus let $c_i$, $1\leq i\leq r$ be the diagonal matrix with $j$-th entry $\left\{\begin{array}{ll} \zeta^{-j} & \mbox{if $(i-1)m<j\leq im$,} \\ 0 & \mbox{otherwise.}\end{array}\right.$
Then $\{ c_1,\ldots ,c_r\}$ is a basis for ${\mathfrak c}$, and the description of $W_1$ follows immediately.
We claim first of all that $W_{\mathfrak c}\supset G(2m,2,r)$.
For this we may clearly assume that $n=mr$.
But now we can apply the argument in Lemma \ref{sln} to show that ${\cal T}\subset T'_m$, and hence by Lemma \ref{criter}, $W_{\mathfrak c}=W_1$.
This proves (a) if $n=mr$.
Suppose therefore that $n>mr$.
Since an element of $N_{{\rm O}(2m,k)}(T)$ which corresponds to a product of $m$ negative 1-cycles in $\overline{W}$ has determinant $(-1)$, it is easy to see that an element of $W_1$ which acts as $-1$ on $c_1$ and $1$  on all $c_i$, $i\geq 2$ has a representative in $W_{\mathfrak c}$ if and only if $Z_{{\rm O}(n,k)}({\mathfrak c})^\theta$ contains an element of determinant $-1$.
This proves (a).

Suppose therefore that $m$ is even.
By Lemma \ref{cycles}, $\theta=\Int n_w$, where $n_w\in N_{{\rm O}(2n,k)}(T)$ and $w=n_wT\in{\overline{W}}$ is either a product of $r$ positive $m$-cycles, a product of $r$ negative $m/2$-cycles, or a product of $r$ negative $m/2$-cycles and one negative 1-cycle.
Constructing a basis for ${\mathfrak c}$ as above, it is easy to see that $W_1=G(m,1,r)$ unless $n=mr/2$ and $w$ is a product of $r$ negative $(m/2)$-cycles, in which case $W_1=G(m,2,r)$.
If $w$ is a product of positive $m$-cycles then $n_w^m=-I$ and hence $\theta$ is $\Aut G$-conjugate to $\Int t$ for a diagonal matrix $t$ with entries of the form $\xi^l$, where $\xi$ is a square-root of $\zeta$ and $l$ is odd.
It follows that $G^\theta$ is isomorphic to a product of subgroups of the form $\GL(r_i)$ ($r_i\geq r$), and hence is connected.
Thus we can apply Lemma \ref{criter} in this case.
Reducing to the case $n=mr$, the argument in the proof of Lemma \ref{sln} shows that $W_{\mathfrak c}=W_1$.

For the case $n_w^m=I$, $G^\theta$ has two irreducible components and therefore we cannot apply Lemma \ref{criter} directly.
We claim first of all that $W_{\mathfrak c}\supset G(m,2,r)$.
For this, we can clearly reduce to the case $n=mr/2$.
But now $T_m$ is trivial, and hence we can lift $\theta$ to an automorphism of $\Spin(2n,k)$ (by Lemma \ref{cover}) and apply Lemma \ref{criter}.
Thus there remains only the case $n>mr/2$ to deal with.
As for the case of $m$ odd above, we associate a vector $c_i\in{\mathfrak c}$ to each negative $m/2$-cycle $w_i$ in the expression for $w$ such that $w_i(c_j)=\zeta^{\delta_{ij}} c_i$; then $\{ c_i:1\leq i\leq r\}$ forms a basis for ${\mathfrak c}$.
Let $L_1$ and $x_1\in L_1$ be as in the proof of Lemma \ref{cycles}.
Then $x_1$ normalizes $T$ and corresponds to $w_1$.
Conjugating $x_1$ by a suitable element of $t$, we may assume that $\theta(n_1)=n_1$.
Let $\overline{w}_1$ be the image of $w_1$ in $W_1$.
Clearly $x_1$ has determinant $(-1)$, hence $\overline{w}_1$ has a representative in $G^\theta$ if and only if there is some element of $Z_{{\rm O}(2n,k)}({\mathfrak c})^\theta$ of determinant $(-1)$.
This proves that $W_{\mathfrak c}=G(m,2,r)$ if $Z_{{\rm O}(2n,k)}({\mathfrak c})^\theta=Z_G({\mathfrak c})^\theta$.
But suppose there exists some $g\in Z_{{\rm O}(2n,k)}({\mathfrak c})$ such that $gx_1\in G_Z^\theta$.
Then $gx_1\theta(gx_1)^{-1}=g\theta(g^{-1})$.
Since it is evidently impossible that $g\theta(g^{-1})=-I$, $W_{\mathfrak c}^Z=W_{\mathfrak c}$.

Suppose therefore that there exists some element $h\in Z_{{\rm O}(2n,k)}({\mathfrak c})^\theta\setminus Z_G({\mathfrak c})^\theta$.
We note that $H=Z_G({\mathfrak c})^{(1)}\cong \SO(2n-rm,k)$.
Replace $h$ by its semisimple part, which is also in $Z_{{\rm O}(2n,k)}({\mathfrak c})^\theta\setminus Z_G({\mathfrak c})^\theta$.
Thus by Lemma \ref{so}, after multiplying $h$ by some element of $H$ we may assume that $h$ normalizes $T$ and acts on $T\cap H$ as a single negative 1-cycle.
Then $L=Z_G(h)^\circ$ is $\theta$-stable, isomorphic to $\SO(2n-1,k)$, and $\Lie(L)$ contains ${\mathfrak c}$.
Moreover, it is easy to see that $\theta$ acts on $T\cap L$ as a product of $r$ negative $(m/2)$-cycles.
Hence $N_{L(0)}({\mathfrak c})/Z_{L(0)}({\mathfrak c})\cong G(m,1,r)$ by Lemma \ref{typeb}.
We deduce that each element of $W_1$ has a representative in $L(0)\subset G(0)$.
\end{proof}

This is the rest of Vinberg's `Second class', $m$ even, $w$ a product of negative cycles is Type I; $m$ even, $w$ a product of positive $m$-cycles is Type II; $m$ odd is Type II.

\begin{rk}\label{rkd}
(a) Our condition on $Z_{{\rm O}(2n,k)}({\mathfrak c})^\theta$ in (a)(ii) is equivalent to the condition given in \cite{vin}.
There Vinberg determines the properties of an automorphism of $\SO(2n,k)$ of the form $\Int g$ by considering the eigenvalues of $g$.
Note that $g^m=\pm I$; let ${\cal S}$ be the set of $m$-th roots of 1 (resp. $-1$) if $g^m=I$ (resp. $-I$).
Suppose $\theta=\Int n_w$ is of odd order.
After replacing $n_w$ by $-n_w$, if necessary, we may assume that $n_w^m=I$ and hence that ${\cal S}=\{Ê\zeta^i:i\in{\mathbb Z}\}$.
Here $r$ is the integer part of half the minimum multiplicity of an eigenvalue of $g$.
Vinberg's condition for $W_{\mathfrak c}$ to be equal to $G(2m,2,r)$ is that the multiplicity of 1 is exactly equal to $2r$.
Let the multiplicity of $\zeta^i$ be $(2r+s_i)=(2r+s_{m-i})$.
Let $L=Z_G({\mathfrak c})^{(1)}\cong\SO(2(n-rm),k)$; then it is easy to see that $Z_L(n_w)\cong \prod_{i=1}^{(m-1)/2}\GL(s_i,k)\times\SO(s_0,k)$, and hence that there is some element of $Z_{{\rm O}(2n,k)}({\mathfrak c})^\theta\setminus Z_G({\mathfrak c})^\theta$ if and only if $s_0>0$.
Similarly, if $m$ is even and $n_w^m=I$ (Type I) then Vinberg's condition for $W_{\mathfrak c}$ to be equal to $G(m,2,r)$ is that the multiplicity of both 1 and $-1$ in $n_w$ is equal to $r$; this is equivalent to the condition in Lemma \ref{typed}(b)(ii).

(b) In Remarks \ref{ark}, \ref{rkc} and \ref{rkb} we pointed out that if $H$ is a minimal Levi subgroup of $G$ whose derived subgroup $L$ contains $T_1$ then any element of $W_{\mathfrak c}$ has a representative in $L(0)$ and $\theta|_L$ is S-regular.
In fact, in each of those cases the restriction $\theta|_{L}$ is also $N$-regular, that is, ${\mathfrak g}(1)$ contains a regular nilpotent element of ${\mathfrak g}$.
(This will be proved in Sect. 6.)
While here we can always find some Levi subgroup such that the restriction of $\theta$ to the derived subgroup $L$ is $S$-regular, it is not in general true that $\theta|_{L}$ is $N$-regular.
In addition, not every element of $W_{\mathfrak c}$ has a representative in $L$.
The exceptions are the cases (i) $m$ is odd and $W_{\mathfrak c}=G(2m,1,r)$ and (ii) $m$ is even, $w$ is a product of negative cycles and $W_{\mathfrak c}=G(m,1,r)$.
On the other hand, in the second case there is a reductive subgroup which contains ${\mathfrak c}$ in the right way.
To see this, let $L=Z_G(h)^\circ\cong\SO(2n-1,k)$ be as constructed in the final paragraph of the proof of Lemma \ref{typed}.
We can then reduce further to a subgroup isomorphic to $\SO(rm+1,k)$, see Rk. \ref{rkb}.
In fact there is a similar construction for the first case as well: let $h\in Z_{{\rm O}(2n,k)}({\mathfrak c})^\theta\setminus Z_G({\mathfrak c})^\theta$, which we can assume to be semisimple, to normalize $T$ and to act on $Z_G({\mathfrak c})^{(1)}\cap T$ as a single negative 1-cycle by the same argument as at the end of the proof above.
Then we also have $Z_G(h)^\circ\cong\SO(2n-1,k)$, ${\mathfrak c}\subset{\mathfrak g}^h$ and each element of $W_{\mathfrak c}$ has a representative in $Z_G(h)^\circ(0)$.
Thus we can take $L$ to be the subgroup of $Z_G(h)^\circ$ constructed as in Rk. \ref{rkb}.
Here we have $L\cong\SO(2rm+1,k)$.
For the other cases we can take $L$ to be $H^{(1)}$, where $H$ is the minimal Levi subgroup of $G$ whose derived subgroup contains $T_1$.
If $m$ is odd and $W_{\mathfrak c}=G(2m,2,r)$ then we have $L\cong\SO(2rm,k)$; if $m$ is even and $w$ is a product of positive $m$-cycles (that is, $n_w^m=-I$) then $L\cong\SL(rm,k)$; if $m$ is even, $n_w^m=I$ and $W_{\mathfrak c}=G(m,2,r)$ then $L\cong \SO(rm,k)$.
\end{rk}

Before we complete the final (classical) case, we require a preparatory lemma.
Let $\gamma:\SL(n,k)\rightarrow\SL(n,k)$, $g\mapsto {^t}g^{-1}$ and let $\psi=\Int J_n\circ\gamma$.
(Hence $\psi(T)=T$ and $G^\psi=\SO(n,k)$.)
Let $T_+=\{ t\in T\mid \psi(t)=t\}^\circ$ and let $T_-=\{ t\in T\mid \psi(t)=t^{-1}\}^\circ$.

\begin{lemma}\label{outersln}
Let $G=\SL(n,k)$, $n>2$, $\charac k\neq 2$.

(a) Any semisimple outer automorphism of $G$ is conjugate to one of the form $\Int t\circ\psi$, where $t\in T_+$.

(b) Two semisimple outer automorphisms $\theta=\Int g\circ\gamma$, $\sigma=\Int h\circ\gamma$ of $G$ are $\Int G$-conjugate if and only if $g\gamma(g)$ and $h\gamma(h)$ are $G$-conjugate.
\end{lemma}

\begin{proof}
By Steinberg's result on semisimple automorphisms \cite[7.5]{steinberg} any semisimple outer automorphism $\theta$ of $G$ stabilizes a maximal torus of $G$ and a Borel subgroup containing it.
After conjugation we may therefore assume that $\theta(T)=T$ and $\theta(B)=B$, where $B$ is the group of upper triangular matrices of determinant 1.
Since $\theta$ is outer, it follows at once that $\theta=\Int t\circ\psi$ for some $t\in T$.
Moreover, if $s\in T_-$ then $\Int s\circ\psi\circ\Int s^{-1}=\Int s^2\circ\psi$.
Thus we may assume that $t\in T_+$.
This proves (a).
For (b), suppose $\theta$ and $\sigma$ are conjugate.
Then $xg\gamma(x^{-1})=\xi^{-2} h$ for some $x\in G$, $\xi\in k^\times$.
Thus $(\xi x)g\gamma((\xi x)^{-1})=h$, hence we may assume that $\xi=1$.
It follows that $xg\gamma(g)x^{-1}=h\gamma(h)$.
Suppose on the other hand that $g\gamma(g)$ and $h\gamma(h)$ are conjugate.
After conjugating by inner automorphisms of $G$ if necessary we may assume by (a) that $g,h\in T_+$.
But now $g\gamma(g)=g^2$ and $h\gamma(h)=h^2$, and with these assumptions $g^2$ and $h^2$ are in fact ${\rm O}(n,k)$-conjugate.
Now it is easy to see that $g=sh$ for some $s\in T_+\cap T_-$.
Thus $\theta$ is $\Int G$-conjugate to $\sigma$.
\end{proof}

\begin{lemma}\label{mcycles}
(a) Let $G=\GL(l,k)$ and let $g\in N_G(T)$ represent an $l$-cycle in $W$.
Then there exists $t\in T$ such that all but one of the non-zero entries of $tg\gamma(t^{-1})$ is equal to 1.

(b) Suppose $g$ is as in (a), that all but one of the non-zero entries of $g$ is equal to 1 and that $l$ is even.
Then the remaining entry is $-\det g$ and $(g\gamma(g))^{l/2}$ is a diagonal matrix with $l/2$ entries equal to $-\det g$ and $l/2$ entries equal to $-1/\det g$.
\end{lemma}

\begin{proof}
A straightforward calculation.
\end{proof}

Note that if in Lemma \ref{mcycles}(a) all of the non-zero entries of $g$ are equal to $\pm 1$ then $\gamma(g)=g$.
This observation will be useful in the proof of Lemma \ref{slnouter} below.

\begin{lemma}\label{formsln}
Suppose $G=\SL(n,k)$ and $\theta$ is outer.
Then $\theta=\Int n_w\circ\gamma$ for some $n_w\in N_G(T)$.

(a) If $m/2$ is even then $w=n_wT$ is a product of $r$ $m$-cycles and $[(n-rm)/2]$ 2-cycles.

(b) If $m/2$ is odd then either:

(i) $(n_w\gamma(n_w))^{m/2}=-I$ and $w$ is a product of $r$ $m$-cycles and $[(n-rm)/2]$ 2-cycles, or;

(ii) $(n_w\gamma(n_w))^{m/2}=I$ and $w$ is a product of $r$ $(m/2)$-cycles and $[(n-rm/2)/2]$ 2-cycles.
\end{lemma}

\begin{proof}
Since $\Aut G$ is generated over $\Int G$ by $\gamma$, clearly $\theta=\Int n_w\circ\gamma$ for some $n_w\in N_G(T)$.
If $m/2$ is even then $w=w_m\cdot w'_m$, where $w_m$ is a product of $r$ $m$-cycles and $w'_m$ is a product of cycles of length less than $m$.
Since $\gamma$ acts trivially on $W$, we may conjugate $\theta$ by $\Int g$ for a suitable element $g\in N_G(T)$ such that $$w_m=\begin{pmatrix} 1 & \ldots & m \end{pmatrix} \ldots \begin{pmatrix} (r-1)m+1 & \ldots & rm \end{pmatrix}$$
Hence ${\mathfrak c}$ has a basis $\{ c_i: 1\leq i\leq r\}$, where $c_i$ is the matrix with $j$-th diagonal entry: $$\left\{\begin{array}{ll}Ê(-\zeta)^{-j} & \mbox{if }(i-1)m<j\leq im, \\ 0 & \mbox{otherwise.}\end{array}\right.$$
Since $m/2$ is even, the $(-\zeta)^{-j}$ are distinct for distinct $j\in{\mathbb Z}/m{\mathbb Z}$ and hence $Z_G({\mathfrak c})\cong (k^\times)^{rm-1}\times\GL(n-rm,k)$.
Now $L=Z_G({\mathfrak c})^{(1)}\cong\SL(n-rm,k)$ and hence by Lemma \ref{outersln} $\Int n_w$ acts on $T\cap L$ as a product of $[(n-rm)/2]$ 2-cycles.
This proves (a).

Let $m/2$ be odd.
If $g\in \GL(m/2,k)$ represents an $m/2$-cycle then $(g\gamma(g))$ also represents an $m/2$-cycle and is of determinant 1, hence (since $m/2$ is odd) $(g\gamma(g))^{m/2}=(\det g\gamma(g))I_{m/2}=I_{m/2}$.
Returning to the general case, $w=w_m\cdot w_{m/2}\cdot w'_m$, where $w_m$ is a product of $r_1$ $m$-cycles, $w_{m/2}$ is a product of $r_2=r-r_1$ $(m/2)$-cycles, and $w'_m$ is a product of cycles of length less than $m/2$.
Write $w_m=w_1\ldots w_{r_1}$ and $w_{m/2}=w_{r_1+1}\ldots w_r$ and let $c_i\in{\mathfrak c}$ be such that $w_i(c_j)=(-\zeta)^{\delta_{ij}}c_j$.
Suppose $r_2>0$: we claim that $(n_w\gamma(n_w))^{m/2}=I$.
Indeed, we can easily construct a $\theta$-stable subgroup $L_r$ of $G$ which is isomorphic to $\SL(m/2,k)$ and such that $c_r\in\Lie(L_r)$.
Then since $\theta|_{L_r}=\Int n_r\circ\gamma$, where $n_r\in L_r$ and $(n_r\gamma(n_r))^{m/2}=I_{m/2}$, we must have $\theta=\Int xn_r\circ\gamma$, where $x\in Z_G(L_r)$ and therefore $(xn_r\gamma(xn_r))^{m/2}=(x\gamma(x))^{m/2}$ must be equal to the identity matrix.
Suppose on the other hand that $w_m$ is non-trivial.
We claim that in this case $(n_w\gamma(n_w))^{m/2}=-I$.
It will clearly suffice to prove this claim when $r=r_1=1$ and $n=m$.
In this case $n_w\gamma(n_w)$ represents a product of two $(m/2)$-cycles in $N_G(T)$.
Thus $(n_w\gamma(n_w))^{m/2}=\pm I$ by Lemma \ref{mcycles}(b).
But if $(n_w\gamma(n_w))^{m/2}=I$ then $n_w\gamma(n_w)$ is conjugate to $\diag (\zeta^{m-1},\zeta^{m-1},\zeta^{m-2},\ldots,1)$ and therefore by Lemma \ref{outersln}(b) $\theta$ is conjugate to an automorphism $\Int g\circ\gamma$ for some $g\in N_G(T)$ which acts on $T$ as a product of two $(m/2)$-cycles.
Since in this case the rank of $\theta$ is 2, this contradicts the assumption that ${\mathfrak c}$ is maximal.
Thus either $w=w_m\cdot w'_m$ or $w=w_{m/2}\cdot w'_m$.
In either case one can apply the argument used in the first paragraph to show that $w'_m$ is a product of 2-cycles as indicated in the Lemma.
\end{proof}

\begin{rk}
(a) If $m/2$ is odd and $w$ is a product of $r$ $m$-cycles and $[(n-rm)/2]$ 2-cycles, the argument in the first part of the proof shows that $Z_G({\mathfrak c})^{(1)}=(\SL(2,k)^{m/2})^r\times\SL(n-rm,k)$.
It is an easy exercise to check in this case that the condition $(n_w\gamma(n_w))^{m/2}=-I$ implies that $\theta$ acts as a zero rank automorphism on the part which is isomorphic to $(\SL(2,k)^{m/2})^r$.

(b) In Vinberg's classification, this is the Fourth case: $m/2$ even is Type III; $m/2$ odd, $w$ a product of $m$-cycles and 2-cycles is Type II, and $m/2$ odd, $w$ a product of $(m/2)$-cycles and 2-cycles is Type I.
\end{rk}

We recall that an automorphism of $\SL(n,k)$ has a unique extension to an automorphism of $\GL(n,k)$ unless $n=2$ (\cite[Lemma 1.4(ii)]{invs}).
In the following lemma, we abuse notation and use $\theta$ to denote the automorphism of $\GL(n,k)$ induced by the action of $\theta$ on $\SL(n,k)$.
(This only appears here for $n>2$ unless $\theta$ is of zero rank.)

\begin{lemma}\label{slnouter}
Let $G=\SL(n,k)$ and let $\theta$ be outer.

(a) If $m/2$ is odd then $W_{\mathfrak c}=W_{\mathfrak c}^Z=W_1=G(m/2,1,r)$.

(b) If $m/2$ is even then $\theta=\Int n_w\circ\gamma$ where $n_w\in N_G(T)$.
We have $W^Z_{\mathfrak c}=W_1=G(m,1,r)$ and $$W_{\mathfrak c}=\left\{\begin{array}{ll} G(m,1,r) & \mbox{if $(n_w\gamma(n_w))^{m/2}=-I$ or $n>mr$ and $Z_{\GL(n,k)}({\mathfrak c})^\theta\neq Z_G({\mathfrak c})^\theta$,} \\ G(m,2,r) & \mbox{otherwise.}\end{array}\right.$$
\end{lemma}

\begin{proof}
Suppose first of all that $m/2$ is odd.
By Lemma \ref{formsln}, $\theta=\Int n_w\circ\gamma$, where $w=n_wT$ is either a product of $r$ $m$-cycles and $[(n-rm)/2]$ 2-cycles, or a product of $r$ $(m/2)$-cycles and $[(n-rm/2)/2]$ 2-cycles.
After conjugation we may assume that $w=w_m\cdot w_2$, where $$w_m=\left\{\begin{array}{ll} \begin{pmatrix} 1 & \cdots & m \end{pmatrix} \ldots \begin{pmatrix} (r-1)m+1 & \cdots & rm \end{pmatrix} & \mbox{if $(n_w\gamma(n_w))^{m/2}=I$,} \\ \begin{pmatrix} 1 & \cdots & m/2 \end{pmatrix} \ldots \begin{pmatrix} (r-1)m/2+1 & \cdots & rm/2 \end{pmatrix} & \mbox{if $(n_w\gamma(n_w))^{m/2}=-I$.}\end{array}\right.$$
We can therefore choose a basis $\{ c_1,\ldots ,c_r\}$ for ${\mathfrak c}$: let $c_i$ be the diagonal matrix with $j$-th entry equal to $\left\{ \begin{array}{ll} (-\zeta)^{-j} & \mbox{if $(i-1)m'<j\leq im'$,} \\ 0 & \mbox{otherwise,} \end{array}\right.$ where $m'=m$ if $w_m$ is a product of $m$-cycles, and $m'=m/2$ if $w_m$ is a product of $m/2$-cycles.
With this description it is clear that $W_1=G(m/2,1,r)$ in either case.
It will therefore suffice to prove that $W_1=W_{\mathfrak c}$ when $n=mr$ for the first case, or $n=mr/2$ for the second.
The second case is a trivial application of Lemma \ref{criter} since here $T_m$ is trivial.
Hence suppose $n=mr$ and $w$ is a product of $r$ $m$-cycles.
It follows from Lemma \ref{mcycles} that after conjugating we may assume $\gamma(n_w)=n_w$.
(This is no longer true if $n>rm$.)
But then $n_w^{m/2}\in G^\theta$, $\Ad n_w^{m/2}$ is trivial on ${\mathfrak c}$ and $L=Z_{\GL(n,k)}(n_w^{m/2})=L_1\times L_2$, where $L_1\cong L_2\cong \GL(rm/2,k)$.
Since $n_w^m=-I$, $n_w^{m/2}$ defines a non-degenerate skew-symmetric form on $k^n$ and hence $H=G^{\Int n_w^{m/2}\circ\gamma}\cong\Sp(rm,k)$.
Thus $H^\gamma=L^\gamma\cong\GL(rm/2,k)$.
We deduce that $\theta|_L$ maps $L_1$ isomorphically onto $L_2$ and vice versa.
We shall show that the little Weyl group for $\theta|_L$ is equal to $G(m/2,1,r)$, hence the same is true for $G$ by the description of $W_1$ above.
But here it is easy to see that the little Weyl group for $\theta|_L$ is isomorphic to the little Weyl group for $\theta^2|_{L_1}$.
We have $n_w\in L$ and therefore we can define the projection of $n_w^2$ onto $L_1$.
Then, since $n_w$ is conjugate to a diagonal matrix with $r$ entries equal to $\xi^{2i-1}$ for each $i\in{\mathbb Z}/m{\mathbb Z}$ (where $\xi$ is a square-root of $\zeta$), the projection of $n_w^2$ onto $L_1$ is conjugate to a diagonal matrix with $r$ entries equal to $\zeta^{2i+1}$ for each $i$, $0\leq i\leq m/2-1$.
It follows that $\Int n_w^2|_{T\cap L_1}$ acts as a product of $r$ $m/2$-cycles.
Let ${\mathfrak c}_1$ be the projection of ${\mathfrak c}\subset\Lie(L_1)\oplus\Lie(L_2)$ onto $\Lie(L_1)$: then by the above remarks $N_{L(0)}({\mathfrak c})/Z_{L(0)}({\mathfrak c})\cong N_{L_1^{\theta^2}}({\mathfrak c}_1)/Z_{L_1^{\theta^2}}({\mathfrak c}_1)$.
(We remark that $L_1^{\theta^2}$ is connected since $\theta^2|_{L_1}=\Int n_w^2|_{L_1}$.)
But the latter group is equal to $G(m/2,1,r)$ by Lemma \ref{sln}.
Hence $W_{\mathfrak c}=G(m/2,1,r)$.

Suppose therefore that $m/2$ is even, hence $\theta=\Int n_w\circ\gamma$ where $w=n_wT$ is a product of $r$ $m$-cycles and $[(n-rm)/2]$ 2-cycles.
By a similar argument to that above, it is straightforward to see that $W_1=G(m,1,r)$.
We claim first of all that $W_{\mathfrak c}^Z=W_1$.
For this we will apply the criterion of Lemma \ref{criter}(b), for the purposes of which we can reduce to the case $r=1$.
We may assume after suitable conjugation that $w=\begin{pmatrix} 1 & \ldots & m \end{pmatrix}$.
Now $T_m$ is the set of matrices of the form $\diag (t,t^{-1},\ldots ,t^{-1})$ and thus ${\cal T}$ is generated by $\diag (\zeta,\zeta^{-1},\ldots ,\zeta^{-1})$.
Let $s=\diag (\zeta,\zeta^{-1},\zeta^3,\zeta^{-3},\ldots ,\zeta^{-1},\zeta)$; then $s\theta(s^{-1})=\diag(\zeta^2,1,\ldots ,\zeta^2,1)$ and hence, by Lemma \ref{criter}, $W_{\mathfrak c}^Z=W_1$.

We now claim that $W_{\mathfrak c}\supset G(m,2,r)$ if $(n_w\gamma(n_w))^{m/2}=I$ and $W_{\mathfrak c}=G(m,1,r)$ if $(n_w\gamma(n_w))^{m/2}=-I$.
It will clearly suffice to prove this when $n=rm$.
Hence by Lemma \ref{mcycles} we may assume that $n_w\in N_G(T)^\gamma$.
Then $\phi=\Int n_w^{m/2}\circ\gamma$ commutes with $\theta$ and therefore $L=G^\phi$ is $\theta$-stable (and connected, by a result of Steinberg \cite[8.1]{steinberg}).
Moreover, it is easy to see that ${\mathfrak c}\subset{\mathfrak l}=\Lie L$ and that $L\cong\SO(rm,k)$ (if $n_w^m=I$) or $L\cong\Sp(rm,k)$ (if $n_w^m=-I$).
Examining the possibilities in Lemma \ref{cycles}, we see that $\theta$ must act on $T\cap L$ as a product of $r$ negative $(m/2)$-cycles.
Thus $$N_{L(0)}({\mathfrak c})/Z_{L(0)}({\mathfrak c})=\left\{\begin{array}{ll} G(m,1,r) & \mbox{if $L=\Sp(rm,k)$,} \\ G(m,2,r) & \mbox{if $L=\SO(rm,k)$,} \end{array}\right.$$ by Lemma \ref{typec} and Lemma \ref{typed}.
This proves our claim.

Suppose therefore that $(n_w\gamma(n_w))^{m/2}=I$ and that $n>mr$.
Let $w_i$ ($1\leq i\leq r$) be the distinct $m$-cycles in the expression for $w=n_wT$ and let $\{ c_i:1\leq i\leq r\}$ be basis for ${\mathfrak c}$ such that $w_i(c_j)=(-\zeta)^{\delta_{ij}}c_j$.
After conjugation we may assume that $w_i=\begin{pmatrix} (i-1)m+1 & \cdots & im \end{pmatrix}$.
But now, since $(n_w\gamma(n_w))^{m/2}=I$ each $m\times m$ submatrix of $n_w$ corresponding to one of the $w_i$ has determinant $-1$.
It is therefore easy to see that $W_{\mathfrak c}=G(m,1,r)$ if and only if there is some element of $Z_{\GL(n,k)}({\mathfrak c})^\theta$ of determinant $-1$.
(Since any element of $\GL(n,k)^\theta$ is of determinant $\pm 1$, this is equivalent to the statement in the Lemma.)
\end{proof}

\begin{rk}\label{rka}
(a) Our condition on $Z_{\GL(n,k)}({\mathfrak c})^\theta$ in (b) is equivalent to that given by Vinberg in \cite[\S 7]{vin}.
This is the Fourth case of Vinberg's classification; $m/2$ even is `Type III'.
Vinberg determines properties of an outer automorphism of $\SL(n,k)$ of the form $\Int g\circ\gamma$ by consideration of the eigenvalues of $g\gamma(g)$ (cf. Lemma \ref{outersln}).
The condition $(n_w\gamma(n_w))^{m/2}=I$ implies that the eigenvalues of $n_w\gamma(n_w)$ are contained in the set ${\cal S}=\{ \zeta^{2i}:i\in{\mathbb Z}\}\supset\{\pm 1\}$.
(This explains the condition $\pm 1\in{\cal S}$ in \cite[p.485]{vin}.)
Here $r$ is the integer part of half the minimal multiplicity of $\lambda\in{\cal S}$ in $n_w\gamma(n_w)$.
Then Vinberg's condition for $W_{\mathfrak c}$ to be equal to $G(m,2,r)$ is that the multiplicity of $1$ is exactly $2r$.
For $i\in{\mathbb Z}/(m/2){\mathbb Z}$ let $2r+s_i$ be the multiplicity of $\zeta^{2i}$ in $n_w\gamma(n_w)$.
(Then $s_{m/2-i}=s_i$ and $s_{m/4}$ is even.)
A direct calculation shows that $Z_G({\mathfrak c})^\theta=\prod_{i=1}^{m/4-1}\GL(s_i,k)\times\Sp(s_{m/4},k)\times\SO(s_0,k)\times (k^\times)^{rm}$ and $Z_{\GL(n,k)}({\mathfrak c})^\theta=\prod_{i=1}^{m/4-1}\GL(s_i,k)\times\Sp(s_{m/4},k)\times{\rm O}(s_0,k)\times(k^\times)^{rm}$, hence $W_{\mathfrak c}=G(m,2,r)$ if and only if $s_0=0$.

(b) In common with type $D$, there is in general no Levi subgroup of $G$ whose derived subgroup $L$ contains $T_1$, such that each element of $W_{\mathfrak c}$ has a representative in $L(0)$ and $\theta|_{L^{(1)}}$ is $N$-regular.
In fact, there is only such a Levi subgroup if $m/2$ is odd: if $w$ is a product of $r$ $m/2$-cycles then one can take the derived subgroup $L$ of a standard Levi subgroup such that $L\cong\SL(rm/2,k)$; if $w$ is a product of $r$ $m$-cycles then the derived subgroup of the group $L$ constructed in the first paragraph is the required group.
Moreover, the above proof does show that if $m/2$ is even then there is a reductive subgroup of $G$ which has the properties we desire.
If $(n_w\gamma(n_w))^{m/2}=-I$ or if $W_{\mathfrak c}=G(m,2,r)$ then let $L$ be the subgroup constructed in the third paragraph of the proof; if $(n_w\gamma(n_w))^{m/2}=I$ then $L\cong\SO(rm,k)$, if $(n_w\gamma(n_w))^{m/2}=-I$ then $L\cong\Sp(rm,k)$.
Suppose that $(n_w\gamma(n_w))^{m/2}=I$ and $W_{\mathfrak c}=G(m,1,r)$.
By the discussion in (a) (following Vinberg) this is true if and only if the multiplicity of $1$ in $n_w\gamma(n_w)$ is greater than $2r$.
But then, since $w$ is a product of $r$ $m$-cycles and $[(n-rm)/2]$ 2-cycles, we can clearly reduce to one of two cases: that $n=rm+2$ and the multiplicity of $1$ is $2r+2$, or that $n=rm+1$ and the multiplicity of $1$ is $2r+1$.
In the first case we can now choose an element $g\in Z_{\GL(n,k)}({\mathfrak c})^\theta\setminus Z_G({\mathfrak c})$ of order 2 such that any element of $W_{\mathfrak c}$ has a representative in $Z_G(g)^{(1)}\cong\SL(rm+1,k)$.
(We can assume $g\in N_G(T)$ and that $g$ represents a 2-cycle in $W$.)
Thus we are reduced to the case $n=rm+1$, where the multiplicity of $1$ in $n_w\gamma(n_w)$ is $2r+1$.
By Lemma \ref{mcycles} we may assume that $\gamma(n_w)=n_w$.
Hence $L=G^{\Int n_w^{m/2}\circ\gamma}\cong\SO(rm+1,k)$.
But ${\mathfrak c}\subset\Lie(L)$ and by Lemma \ref{cycles}, $\theta$ acts on $T\cap L$ as a product of $r$ negative $m/2$-cycles.
It follows by Lemma \ref{typeb} that $N_{L(0)}({\mathfrak c})/Z_{L(0)}({\mathfrak c})=G(m,1,r)$.
Now $L=\SO(rm+1,k)$ is a subgroup of $G$ whose Lie algebra contains ${\mathfrak c}$, which has the same little Weyl group as $G$, and such that $\theta|_L$ is $S$-regular.
These subgroups $L$ constructed in this way will be very useful to us in Sect. \ref{kw}.
\end{rk}

Lemmas \ref{sln}-\ref{slnouter} provide a new proof of Vinberg's description of the little Weyl group for classical graded Lie algebras \cite[Prop.s 15 \& 16]{vin}.
We deduce from Lemmas \ref{sln}, \ref{typec}, \ref{typed} and \ref{slnouter} that any classical graded Lie algebra is saturated except for an outer automorphism of $\SL(n,k)$ of order divisible by 4 for which $W_{\mathfrak c}=G(m,2,r)$ (cf. \cite[Prop. 16]{vin}).
Moreover, we remark that $W_{\mathfrak c}^Z=W_1$ unless (i) $\theta$ is an odd order automorphism of $\SO(2n,k)$, $n>mr$ and $Z_{{\rm O}(2n,k}({\mathfrak c})^\theta=Z_G({\mathfrak c})^\theta$ or (ii) $\theta=\Int g$ is an even order automorphism of $\SO(2n,k)$, $g^m=I$, $n>mr/2$ and $Z_{{\rm O}(2n,k)}({\mathfrak c})^\theta=Z_G({\mathfrak c})^\theta$.
(In the notation of \cite{vin}, these cases are: (i) the Second case, Type III where $n>mr$ and $V'(\pm 1)=0$, (ii) the Second case, Type I where $n>mr/2$ and $V'(\pm 1)=0$.)

Our proof of the description of $W_{\mathfrak c}$ is significantly longer than that in \cite{vin}.
However, as indicated by Remarks \ref{ark}, \ref{rkc}, \ref{rkb}, \ref{rkd} and \ref{rka}, this alternative perspective on the little Weyl group provides a relatively easy way to establish the existence of a KW-section for all classical graded Lie algebras (see Sect. \ref{kw}).

To prove that $k[{\mathfrak c}^W]$ is generated by pseudoreflections we apply the following result of Panyushev \cite[Thm. \& Prop. 2]{panorbits}.

{\it - Let $U\subset V$ be vector spaces, let $G\subset\GL(V)$ be a connected reductive group, let $W\subset\GL(U)$ be a finite group of order coprime to $\charac k$ such that $V\quot G\cong U/W$.
Then $W$ is generated by pseudoreflections.}

To prove that the order of $W_{\mathfrak c}$ is coprime to $\charac k$, we apply Carter's results on conjugacy classes in Weyl groups.
Let us briefly recall the set-up.
Thus let $W$ be an arbitrary Weyl group with natural complex representation $V$.
Assume that the root system associated to $W$ is irreducible.
Any element $w\in W$ can be expressed as a product $w=w_1w_2$, where $w_1^2=w_2^2=1$ and $\{ v\in V\,|\, w_1\cdot v=-v\}\cap\{ v\in V\,|\, w_2\cdot v=-v\}=\{0\}$.
Moreover, any involution $w'$ in $W$ can be expressed as a product of reflections corresponding to $l(w')$ orthogonal roots.
Thus the expression $w=w_1w_2$ gives subsets $I_1,I_2$ of the root system $\Phi$ such that $w_i=\prod_{\alpha\in I_i}s_\alpha$ for $i=1,2$.
Moreover, $\#(I_i)=l(w_i)$ for $i=1,2$ and $l(w_1)+l(w_2)=l(w)$.
One associates a graph to $w$ with one node for each $\alpha\in I_1$ and one node for each $\beta\in I_2$, with $\langle\alpha,\beta\rangle\langle\beta,\alpha\rangle$ edges between nodes corresponding to {\it distinct} roots $\alpha,\beta\in I_1\cup I_2$ (since $I_1$ and $I_2$ may not be disjoint).
The graph $\Gamma$ so constructed is uniquely defined by $w$.
For example, if $\Gamma$ is the Dynkin diagram on the root system associated to $W$ then $w$ is a Coxeter element of $W$; if $\Gamma$ is the trivial graph then $w=1$.
If $\Gamma=\Gamma'\cup \Gamma''$, where $\Gamma'$ and $\Gamma''$ are orthogonal subgraphs, then there is a corresponding decomposition $w=w'w''$ and orthogonal root subsystems $\Phi',\Phi''\subset\Phi$ such that $w'\in W(\Phi')$, $w''\in W(\Phi'')$.
Thus the decomposition of $\Gamma$ into irreducible components gives a corresponding decomposition of $w$ as a product of commuting elements.
The irreducible graphs $\Gamma$ which can appear via this construction are listed in \cite[Table 2, p.10]{carter}; the characteristic polynomials (in the natural representation), and hence the orders of such elements are given in \cite[Table 3, p.23]{carter}.
Tables 7-11 at the end of \cite{carter} give a classification of conjugacy classes in the exceptional type Weyl groups.
A few words of explanation of the symbols which we use in the proof below: $\Gamma$ is the graph associated to $w$ as detailed above, which we refer to as the type of $w$; $W_1$ is the (Weyl group of the) minimal root subsystem of $\Phi$ containing all roots $\alpha$ associated to nodes in $\Gamma$ and $W_2$ is the (Weyl group of the) subsystem of all roots in $\Phi$ which are orthogonal to $W_1$.
If $w\in W$ is of order $m$, then it is easy to see that the reflection corresponding to an element of $W_2$ acts trivially on $\{ v\in V\,| w\cdot v=\zeta v\}$.

\begin{proposition}\label{poly}
Let $G$ be almost simple.

(a) $W_{\mathfrak c}$ is generated by pseudoreflections.

(b) $k[{\mathfrak c}]^{W_{\mathfrak c}}$ is a polynomial ring.
\end{proposition}

\begin{proof}
This is known for $m=2$ by \cite[4.11]{invs} (part (a) also follows from Lemma \ref{thetasplit} and \cite[\S 4]{richinvs}); hence assume $m\geq 3$.
For the classical graded Lie algebras, (a) is true by Lemmas \ref{sln}, \ref{typec}, \ref{typeb}, \ref{typed} and \ref{slnouter}.
If $G$ is of exceptional type, or if $G$ is of type $D_4$ and $\theta$ is an outer automorphism such that $\theta^3$ is inner, then we claim that $W_1$ is of order coprime to $p$.
Indeed, if $G$ is of exceptional type then the assumption that $p$ is good implies that $p$ is coprime to the order of $W$ except in the following cases:

(i) $G$ is of type $E_6$ and $p=5$,

(ii) $G$ is of type $E_7$ and $p=5$,

(iii) $G$ is of type $E_7$ and $p=7$,

(iv) $G$ is of type $E_8$ and $p=7$.

In type $E_6$, $\theta$ could of course be outer.
In this case let $\gamma$ be an involutive automorphism of $G$ satisfying $\gamma(t)=t^{-1}$ for all $t\in T$; then any outer automorphism of $G$ is of the form $\Int g\circ\gamma$ for some $g\in G$, hence $\theta=\Int n_w\circ \gamma$ for some $n_w\in N_G(T)$.
Moreover, the induced action of $\gamma$ on $W$ is trivial and thus $W^\theta=Z_W(w)$, where $w=n_wT$.
Hence in type $E_6$ it will suffice to show that centralizers have order prime to $p$.
An inspection of Tables 9-11 in \cite{carter} shows that $p\nmid Z_W(w)$ except for the following cases:

(i) $w$ is of type $A_1$, $A_4$, or $A_4\times A_1$.
The first case is an involution.
In the other two cases, $w$ is of order divisible by 5 and thus can only appear if $p>5$.

(ii) $w$ is of type $A_1$, $A_2$, $A_4$, $A_4\times A_1$, $A_1^6$, $A_4\times A_2$, $D_6$, $A_1^7$, $D_4\times A_1^3$, $D_6\times A_1$ or $E_7(a_3)$.
In the three cases $A_1$, $A_1^6$, and $A_1^7$, $\theta$ is an involution.
If $w$ is of type $A_4$, $A_4\times A_1$, $A_4\times A_2$, $D_6$, $D_6\times A_1$ or $E_7(a_3)$ then its order is divisible by 5.
If $w$ is of type $A_2$ then $Z_W(w)$ is of order $2^5.3^3.5$.
However, here $W_2=A_5$, and hence $Z_{W^\theta}({\mathfrak c})$ is of order divisible by $6!=2^4.3^2.5$.
It follows that $W_1$ has order dividing $6$.
Finally, if $w$ is of type $A_4\times A_2$ then $\theta$ is of order (divisible by) 15 and hence ${\mathfrak c}$ is trivial.

(iii) $w$ is of type $A_6$, $A_1^7$ or $E_7(a_1)$.
The case $A_1^7$ is an involution.
If $w$ is of type $A_6$ or $E_7(a_1)$ then $\theta$ is of order divisible by 7, which contradicts the assumption that $m$ is coprime to $p$.

(iv) $w$ is of type $A_1$, $A_6$, $A_1^7$, $A_6\times A_1$, $E_7(a_1)$, $A_1^8$ or $D_8$.
The cases $A_1$ and $A_1^7$ are involutions.
If $w$ is of type $A_6$, $A_6\times A_1$, $E_7(a_1)$ or $D_8$ then the order of $\theta$ is divisible by 7, which contradicts the assumption on $m$.

Finally, suppose $G$ is of type $D_4$ and $\theta$ is an outer automorphism such that $\theta^3$ is inner.
Then $p>3$.
But $W(D_4)$ has order $192=2^6.3$.
Hence $W^\theta$ has order coprime to $p$.
Thus $W$ is in all cases generated by pseudoreflections by Panyushev's theorem.
This proves (a).
But (b) now follows since $W_{\mathfrak c}$ has order coprime to $p$ unless ($\theta$ is an involution or) ${\mathfrak g}$ is a classical graded Lie algebra.
In the classical case $W_{\mathfrak c}$ is one of $G(m',1,r)$, $G(m',2,r)$ where $m'$ is coprime to $p$, hence it is easily verified that $k[{\mathfrak c}]^{W_{\mathfrak c}}$ is a polynomial ring.
\end{proof}

\begin{theorem}
Suppose $G$ satisfies the standard hypotheses.
Then $k[{\mathfrak g}(1)]^{G(0)}$ is a polynomial ring.
\end{theorem}

\begin{proof}
This follows from the construction of $\hat{\mathfrak g}=\tilde{\mathfrak g}\oplus{\mathfrak t}_0={\mathfrak g}\oplus{\mathfrak t}_1$ in Prop. \ref{reduction}.
Let ${\mathfrak c}$ be a Cartan subspace of ${\mathfrak g}$ and let ${\mathfrak c}_1=\{ t\in {\mathfrak t}_1\,|\, d\hat\theta(t)=\zeta t\}$.
Clearly $\hat{\mathfrak c}={\mathfrak c}\oplus{\mathfrak c}_1$ is a Cartan subspace of $\hat{\mathfrak g}$.
In fact it is the unique Cartan subspace of $\hat{\mathfrak g}$ which contains ${\mathfrak c}\cap\Lie(G^{(1)})$.
Let $\tilde{\mathfrak c}=\hat{\mathfrak c}\cap\tilde{\mathfrak g}$.
Then $\hat{\mathfrak c}=\tilde{\mathfrak c}\oplus{\mathfrak c}_0$, where ${\mathfrak c}_0=\{ t\in{\mathfrak t}_0\,|\,d\hat\theta(t)=\zeta t\}$.
Clearly $k[\hat{\mathfrak c}]^{W_{\mathfrak c}}=k[\tilde{\mathfrak c}]^{W_{\mathfrak c}}\otimes k[{\mathfrak c}_0]$.
It is easy to see that Prop. \ref{poly} extends to a product of almost simple groups and groups of the form $\GL(n,k)$, hence to $\tilde{G}$.
Thus $k[\tilde{\mathfrak c}]^{W_{\mathfrak c}}$ is a polynomial ring.
It follows that $k[\hat{\mathfrak c}]^{W_{\mathfrak c}}=k[{\mathfrak c}]^{W_{\mathfrak c}}\otimes k[{\mathfrak c}_1]$ is also a polynomial ring.
Now let $J_1$ be the maximal ideal of all positive degree elements of $k[{\mathfrak c}_1]$: then $J_1$ is generated by elements of degree 1 and hence its set of zeros is a hyperplane in $\hat{\mathfrak c}/{W_{\mathfrak c}}$ (identifying $\hat{\mathfrak c}/W_{\mathfrak c}$ with a vector space of dimension $r+\dim{\mathfrak t}_1$).
But therefore $k[{\mathfrak c}]^{W_{\mathfrak c}}\cong k[\hat{\mathfrak c}]^{W_{\mathfrak c}}/J_1k[\hat{\mathfrak c}]^{W_{\mathfrak c}}$ is a polynomial ring.
Thus the result follows by Thm. \ref{chev}.
\end{proof}

\section{Kostant-Weierstrass slices}\label{kw}

A long-standing conjecture in this field (originally stated in characteristic zero \cite[no. 7]{popov}) is the existence of a KW-section in ${\mathfrak g}(1)$ to the invariants.
(For details on Weierstrass slices see \cite[\S 8]{vinpopov} or \cite{popovsections} for more recent work.
In the case of a periodically graded reductive Lie algebra, Panyushev \cite{panslice} introduced the terminology of Kostant-Weierstrass slice or KW-section because of the analogy with Kostant's slice to the regular conjugacy classes in ${\mathfrak g}$.)

\begin{definition}
A Kostant-Weierstrass slice or KW-section for $\theta$ is a linear subvariety ${\mathfrak v}$ of ${\mathfrak g}(1)$ such that the embedding ${\mathfrak v}\hookrightarrow{\mathfrak g}(1)$ induces an isomorphism of affine varieties ${\mathfrak v}\rightarrow{\mathfrak g}(1)\quot G(0)$.
\end{definition}

The prototype is Kostant's slice $e+{\mathfrak z}_{\mathfrak g}(f)$ in ${\mathfrak g}$, where $\{ h,e,f\}$ is an $\mathfrak{sl}(2)$-triple such that $e$ is a regular nilpotent element.
The case $m=2$ is also known (\cite{kostrall} in characteristic zero, \cite{invs} in positive characteristic).
Essentially, one can reduce the involution case to the $m=1$ case by constructing a reductive subalgebra of ${\mathfrak g}$ for which a Cartan subspace of ${\mathfrak g}(1)$ is a Cartan subalgebra.
One can then apply the usual construction since an involution is $S$-regular if and only if it is $N$-regular.
(Recall that $\theta$ is $S$-regular (resp. $N$-regular) if ${\mathfrak g}(1)$ contains a regular semisimple (resp. nilpotent) element of ${\mathfrak g}$.)
Applying such an argument in the general case is problematic since a general finite-order automorphism can be $S$-regular but not $N$-regular, and vice versa.
On the other hand, it is known due to Panyushev (in characteristic zero) that an $N$-regular automorphism always admits a KW-section \cite{panslice}.
(Earlier Panyushev also showed that if $G(0)$ is semisimple then $\theta$ admits a KW-section \cite{pansemislice}.)
The slice constructed in \cite{panslice} is a natural choice: one chooses $e\in{\mathfrak g}(1)$ to be a regular nilpotent element of ${\mathfrak g}$, embeds $e$ in an $\mathfrak{sl}(2)$-triple $\{ h,e,f\}$ with $h\in{\mathfrak g}(0)$ and $f\in{\mathfrak g}(-1)$, and sets ${\mathfrak v}=e+{\mathfrak z}_{{\mathfrak g}(1)}(f)$.
We will show in this section that Panyushev's theorem can be applied to the case of a classical graded Lie algebra (under the assumption of the standard hypotheses) by fairly straightforward reduction to (certain) $N$-regular cases.
Indeed, almost all of the work required has been carried out in the previous section.
Recall from Remarks \ref{ark}, \ref{rkc}, \ref{rkb}, \ref{rkd} and \ref{rka} the construction of the semisimple subgroup $L$ such that ${\mathfrak c}\subset\Lie(L)$, each element of $W_{\mathfrak c}$ has a representative in $L(0)$, and $\theta|_L$ is $S$-regular.
(The analysis of the Weyl group in Sect. 4 clearly goes through in exactly the same way if the characteristic of $k$ is zero.)
The form of $L$ for each case is summed up in Table 1.
(A few words on the entries in the table.
We use Vinberg's classification, hence, for example, the `Second case', Type II is denote 2II.
Let $m_0=m/2$ in the Fourth case, and let $m_0=m$ otherwise.
The column marked `$\theta$' gives conditions on $g$, where $\theta$ is of the form $\Int g$ or $\Int g\circ\gamma$.
In the column marked $L$ we have placed a star next to the entries of the form $\SL(rm,k)$, $\SL(rm/2,k)$, $\SL(rm/2,k)^2$ since if $p|r$ these should be replaced with the corresponding general linear group.
This is always possible, if we assume $G$ is not equal to $\SL(V)$ where $p|\dim V$.
On the other hand, it is clearly also possible if $G=\GL(V)$.
In the column $\theta|_L$ we have marked the entry for 4II with a double star since here $L\cong\SL(rm/2,k)\times\SL(rm/2,k)$ and the action of $\theta$ is given by: $(g_1,g_2)\mapsto (\sigma(g_2),g_1)$, where $\sigma$ is an inner automorphism of $\SL(rm/2,k)$ of order $m/2$, rank $r$.
Thus, strictly speaking, $\sigma$ is the First case and not $\theta$.)

\begin{table}
\begin{center}
\caption{Reduction to the $S$-regular case}
\vline
\begin{tabular}{ccccccc}
\hline
Case & $G$ & $m_0$ & $\theta$ & $W_{\mathfrak c}$ & $L$ & $\theta|_L$ \\
\hline
1 & $\SL(n,k)$ & - & - & $G(m,1,r)$ & $\SL(rm,k)$* & 1 \\
2I & $\SO(2n+1,k)$ & even & $g^m=I$ & $G(m,1,r)$ & $\SO(rm+1,k)$ & 2 \\
2I & $\SO(2n,k)$ & even & $g^m=I$ & $G(m,1,r)$ & $\SO(rm+1,k)$ & 2 \\
 & & & $Z_{\rm{O}(2n,k)}({\mathfrak c})^\theta\neq Z_G({\mathfrak c})^\theta$ & & \\
2I & $\SO(2n,k)$ & even & $g^m=I$ & $G(m,2,r)$ & $\SO(rm,k)$ & 2 \\
 & & & $Z_{\rm{O}(2n,k)}({\mathfrak c})^\theta=Z_G({\mathfrak c})^\theta$ & & \\
2II & $\SO(2n,k)$ & even & $g^m=-I$ & $G(m,1,r)$ & $\SL(rm,k)$* & 1 \\
2III & $\SO(2n+1,k)$ & odd & - & $G(2m,1,r)$ & $\SO(2rm+1,k)$ & 2 \\
2III & $\SO(2n,k)$ & odd & $Z_{\rm{O}(2n,k)}({\mathfrak c})^\theta\neq Z_G({\mathfrak c})^\theta$ & $G(2m,1,r)$ & $\SO(2rm+1,k)$ & 2 \\
2III & $\SO(2n,k)$ & odd & $Z_{\rm{O}(2n,k)}({\mathfrak c})^\theta= Z_G({\mathfrak c})^\theta$ & $G(2m,2,r)$ & $\SO(2rm,k)$ & 2 \\
3I & $\Sp(2n,k)$ & even & $g^m=-I$ & $G(m,1,r)$ & $\Sp(rm,k)$ & 3 \\
3II & $\Sp(2n,k)$ & even & $g^m=I$ & $G(m,1,r)$ & $\SL(rm,k)$* & 1 \\
3III & $\Sp(2n,k)$ & odd & - & $G(2m,1,r)$ & $\Sp(2rm,k)$ & 3 \\
4I & $\SL(n,k)$ & odd & $(g\gamma(g))^{m/2}=I$ & $G(m/2,1,r)$ & $\SL(rm/2,k)$* & 4 \\
4II & $\SL(n,k)$ & odd & $(g\gamma(g))^{m/2}=-I$ & $G(m/2,1,r)$ & $\SL(rm/2,k)^2$* & 1** \\
4III & $\SL(n,k)$ & even & $(g\gamma(g))^{m/2}=I$ & $G(m,1,r)$ & $\SO(rm+1,k)$ & 2 \\
 & & & $Z_{\GL(n,k)}({\mathfrak c})^\theta\neq Z_G({\mathfrak c})^\theta$ & &  \\
4III & $\SL(n,k)$ & even & $(g\gamma(g))^{m/2}=I$ & $G(m,2,r)$ & $\SO(rm,k)$ & 2 \\
 & & & $Z_{\GL(n,k)}({\mathfrak c})^\theta= Z_G({\mathfrak c})^\theta$ & & \\
4III & $\SL(n,k)$ & even & $(g\gamma(g))^{m/2}=-I$ & $G(m,1,r)$ & $\Sp(rm,k)$ & 3 \\
\hline
\end{tabular}\vline
\end{center}
\end{table}\nopagebreak

In characteristic zero, it therefore remains only to show that the pairs $(L,\theta|_L)$ listed in Table 1 are $N$-regular; Panyushev's theorem on $N$-regular automorphisms \cite[Thm. 3.5]{panslice} then implies that any classical graded Lie algebra admits a KW-section.
In positive characteristic, we provide the following generalization of Panyushev's result.
Our proof is broadly similar, although Cor. \ref{dimcor} allows us to avoid a potentially troublesome argument \cite[3.3]{panslice} involving $\mathfrak{sl}(2)$-triples.

\begin{proposition}\label{kwprop}
Let $G$ be a group satisfying the standard hypotheses and let $\theta$ be an automorphism of $G$ of order $m$, $p\nmid m$.
Suppose that $\theta$ is N-regular.
Then the restriction homomorphism $k[{\mathfrak g}]^G\rightarrow k[{\mathfrak g}(1)]^{G(0)}$ is surjective.
Let $e\in{\mathfrak g}(1)$ be a regular nilpotent element of ${\mathfrak g}$ and let $\lambda:k^\times\rightarrow G(0)$ be an associated cocharacter for $e$ (see Rk. \ref{associated}).
Let ${\mathfrak u}$ be a $\lambda(k^\times)$-stable subspace of ${\mathfrak g}(1)$ such that ${\mathfrak u}\oplus[{\mathfrak g}(0),e]={\mathfrak g}(1)$.
Then $e+{\mathfrak u}$ is a KW-section for $\theta$.
\end{proposition}

\begin{proof}
Let ${\mathfrak w}$ be a $\theta$-stable, $\Ad \lambda(k^\times)$-stable subspace of ${\mathfrak g}$ such that ${\mathfrak w}\oplus[{\mathfrak g},e]={\mathfrak g}$.
Recall (see \cite[6.3-6.5]{veldkamp} and \cite[Pf. of Lemma 1]{premettange} in good characteristic) that the embedding $e+{\mathfrak w}\hookrightarrow{\mathfrak g}$ induces an isomorphism $e+{\mathfrak w}\rightarrow{\mathfrak g}\quot G$.
Let $n=\rank G$ and let $F_1,\ldots ,F_n$ be algebraically independent homogeneous generators of $k[{\mathfrak g}]^G$.
Then the differentials $(dF_i)_e|_{\mathfrak w}$ are linearly independent elements and span ${\mathfrak w}^*$, hence their restrictions $(dF_i)_e|_{\mathfrak u}=(dF_i|_{e+{\mathfrak u}})_e$ span ${\mathfrak u}^*$.
Since $\dim{\mathfrak u}=r$ by Cor. \ref{dimcor} and separability of orbits (see, for example \cite[4.2]{invs}), we may assume after renumbering that $(dF_1|_{e+{\mathfrak u}})_e,\ldots ,(dF_r|_{e+{\mathfrak u}})_e$ span ${\mathfrak u}^*$.
In particular, $(dF_1|_{{\mathfrak g}(1)})_e,\ldots (dF_r|_{{\mathfrak g}(1)})_e$ are linearly independent.
Let $u_1,\ldots ,u_r$ be algebraically independent homogeneous generators of $k[{\mathfrak g}(1)]^{G(0)}$, and let $h$ be a monomial in the $u_i$.
Then, since $u_i(e)=0$ for $1\leq i\leq r$, $dh_e=0$ unless $h=u_i$ for some $u_i$.
Thus we can express $F_i|_{{\mathfrak g}(1)}$, $1\leq i\leq r$ (uniquely) as $f_i+g_i$, where $f_i$ is linear in the $u_j$ ($1\leq j\leq r$) and $g_i$ is in the ideal of $k[u_1,\ldots ,u_r]$ generated by all $u_iu_j$, $1\leq i\leq j\leq r$.
Note that we have $(dg_i)_x=0$ for any nilpotent element $x\in{\mathfrak g}(1)$ by Lemma \ref{unstable}.
Therefore $(dF_i|_{{\mathfrak g}(1)})_e=(df_i)_e$ for each $i$, $1\leq i\leq r$ and the differentials $(df_i)_e$ are linearly independent.
It follows that the $f_i$, $1\leq i\leq r$ are algebraically independent and that $k[u_1,\ldots ,u_r]=k[f_1,\ldots ,f_r]$.
Since the $F_i$ are homogeneous, there exist integers $m_i$, $1\leq i\leq r$ such that $F_i(\eta x)=\eta^{m_i}F_i(x)$ for all $\eta\in k$ and all $x\in{\mathfrak g}$.
But now, since the expression $F_i|_{{\mathfrak g}(1)}=f_i+g_i$ is unique, it follows that $f_i(\eta x)=\eta^{m_i}x$ for any $\eta\in k$, $x\in{\mathfrak g}(1)$ and similarly for $g_i$.
After reordering we may assume that $m_1\leq \ldots \leq m_r$.
It now follows that $F_i|_{{\mathfrak g}(1)}\in f_i+k[f_1,\ldots ,f_{i-1}]$.
Thus the restrictions $F_i|_{{\mathfrak g}(1)}$ (resp. $F_i|_{e+{\mathfrak u}}$), $1\leq i\leq r$ are algebraically independent and generate $k[{\mathfrak g}(1)]^{G(0)}$ (resp. $k[e+{\mathfrak u}]$).
This completes the proof
\end{proof}


\begin{proposition}\label{nreg}
Let $G$ be of classical type, let $L$ be the subgroup as listed in Table 1 and let ${\mathfrak l}=\Lie(L)$.
Then $\theta|_L$ is N-regular.
\end{proposition}

\begin{proof}
We may assume that $G=L$.
We have the following possibilities:

(i) $G=\SL(rm,k)$ or $\GL(rm,k)$ (if $p|r$) and $\theta$ is inner, $w$ a product of $m$-cycles;

(ii) $G=\SO(rm+1,k)$ and $w$ is a product of negative $m/2$-cycles ($m$ even);

(iii) $G=\SO(rm,k)$ and $w$ is a product of negative $m/2$-cycles ($m$ even);

(iv) $G=\SO(2rm+1,k)$ and $w$ is a product of positive $m$-cycles ($m$ odd);

(v) $G=\SO(2rm,k)$ and $w$ is a product of positive $m$-cycles ($m$ odd);

(vi) $G=\Sp(rm,k)$ and $w$ is a product of negative $m/2$-cycles ($m$ even);

(vii) $G=\Sp(2rm,k)$ and $w$ is a product of positive $m$-cycles ($m$ odd);

(viii) $G=\SL(rm/2,k)$ or $\GL(rm/2,k)$ (if $p|r$), $\theta$ is outer and $w$ is a product of $m/2$-cycles;

(ix) $G=\SL(rm/2,k)^2$ or $\GL(rm/2,k)^2$ (if $p|r$) and $\theta$ is an automorphism of the form $(g_1,g_2)\mapsto(\sigma(g_2),g_1)$ where $\sigma$ is of order $m/2$ and acts on the maximal torus of diagonal matrices in $\SL(rm/2,k)$ as a product of $r$ $m/2$-cycles.

The argument for the last case clearly reduces immediately to verifying the lemma for $\sigma$, hence to case (i).
Thus we consider the cases (i)-(viii).
To prove N-regularity we replace $\theta$ by an $\Int G$-conjugate such that $(B,T)$ is a fundamental pair for $\theta$.
If $\theta$ is inner and $G=\SL(n,k)$ then this means that $\theta=\Int t$, where $t\in\GL(n,k)$ is diagonal and has $r$ entries equal to $\zeta^i$ for each $i$, $0\leq i\leq m-1$.
But then we can assume after conjugating by an element of the normalizer of $T$ that $t=\diag (\zeta^{m-1},\zeta^{m-2},\ldots ,1)$, in which case the nilpotent element with 1 on the first upper diagonal and zero elsewhere is clearly in ${\mathfrak g}(1)$.
One can carry out similar calculations for the automorphisms of $\SO(2n+1,k)$, $\SO(2n,k)$ and $\Sp(2n,k)$.
For all cases except (iii) with $r$ odd we may assume after conjugation that $\theta=\Int t$, where $t\in T$ is as given in the following list.
Recall that $\zeta$ is a primitive $m$-th root of unity and $\xi$ is a square root of $\zeta$.

(a) $t=\diag(1,\zeta^{m-1},\ldots ,1)$ and $e=\sum_{i=1}^{n}e_{i,i+1}-\sum_{i=n+1}^{2n}e_{i,i+1}$ in cases (ii) ($n=rm/2$) and (iv) ($n=rm$).

(b) $t=\diag (\zeta^{m-1},\zeta^{m-2},\ldots ,1,1,\zeta^{-1},\ldots ,\zeta)$ and $e=\sum_{i=1}^{n-1}e_{i,i+1}+e_{n-1,n+1}-e_{n,n+2}-\sum_{n+1}^{2n-1}e_{i,i+1}$ in cases (iii) with $r$ even ($n=rm/2$) and (v) ($n=rm$).

(c) $t=\diag (\xi^{2m-1},\xi^{2m-3},\ldots ,\xi)$ and $e=\sum_{i=1}^ne_{i,i+1}-\sum_{i=n+1}^{2n-1}e_{i,i+1}$ in cases (vi) ($n=rm/2$) and (vii) ($n=rm$).

To check $N$-regularity for case (iii) with $r$ odd, let $J_2$ be the $2\times 2$ matrix with 1 on the antidiagonal and 0 on the diagonal, and let $s$ be the diagonal $(rm/2-1)\times (rm/2-1)$ matrix with $j$-th entry $-\zeta^{-j}$.
Then after conjugation we may assume that $\theta=\Int \begin{pmatrix}Ês & 0 & 0 \\ 0 & J_2 & 0 \\ 0 & 0 & -s \end{pmatrix}$.
Now $e=\sum_{i=1}^{n-1}e_{i,i+1}+e_{n-1,n+1}-e_{n,n+2}-\sum_{i=n+1}^{2n-1}e_{i,i+1}$ is a regular nilpotent element of ${\mathfrak g}(1)$, where $n=rm/2$.

This leaves only the case where $G=L=\SL(rm/2,k)$ and $\theta$ is outer.
We have $\theta=\Int g\circ\gamma$ and $(g\gamma(g))^{m/2}=I$.
Let $\psi=\Int (tJ_n)\circ\gamma$, where $\gamma:x\mapsto {^t}x^{-1}$, $J_n$ is the matrix with 1 on the antidiagonal and 0 elsewhere, and $$t=\left\{\begin{array}{ll} \diag (\zeta^{(m-2)/4},\zeta^{(m-6)/4},\ldots ,\zeta^{-(m-2)/4}) & \mbox{if $r$ is odd,} \\ \diag(\zeta^{m-1},\zeta^{m-2},\ldots ,1,-\zeta^{-1},-\zeta^{-2},\ldots ,-1) & \mbox{if $r/2$ is even.} \\ \diag(\zeta^{m-1},\zeta^{m-2},\ldots ,-1,\zeta^{-1},\zeta^{-2},\ldots ,-1) & \mbox{if $r/2$ is odd.} \end{array}\right.$$

It is an easy to calculation to see that $tJ_n\gamma(tJ_n)=tJ_n\gamma(t)J_n^{-1}$ has $r$ entries equal to $\zeta^{2i}$ for each $i$, $0\leq i<m/2$ and hence $\theta$ is conjugate to $\psi$ by Lemma \ref{mcycles}.
But $\psi$ is $N$-regular: $$e=\left\{\begin{array}{ll} \sum_{i=1}^{(n-1)/2}e_{i,i+1}-\sum_{(n+1)/2}^{n-1}e_{i,i+1} & \mbox{if $r$ is odd,}Ê\\ \sum_{i=1}^{n/2}e_{i,i+1}-\sum_{n/2+1}^{n-1}e_{i,i+1} & \mbox{if $r$ is even.} \end{array}\right.$$
where $n=rm/2$.
This proves that $\theta$ is $N$-regular in each of the cases concerned.

\end{proof}

We therefore have:

\begin{theorem}
Let $G$ be one of $\SL(n,k)$ ($p\nmid n$), $\GL(n,k)$, $\SO(n,k)$, $\Sp(2n,k)$.
Then $({\mathfrak g},d\theta)$ admits a KW-section.
\end{theorem}

\begin{proof}
This follows immediately from Prop. \ref{kwprop} and Prop. \ref{nreg} and the fact that the universal covering of $\SO(n,k)$ is separable.
\end{proof}

\begin{rk}
(a) In the case where $({\mathfrak g},d\theta)$ is $N$-regular but not $S$-regular (and locally free), our construction shows that there are many different KW-sections.
A trivial example is a zero rank $N$-regular grading: applying Panyushev's theorem directly, one obtains $\{ e\}$ as a KW-section; our construction via the subgroup $L$ gives $\{ 0\}$.

(b) These methods can be applied to prove the existence of KW-sections for exceptional type Lie algebras, as well as the remaining outer automorphisms in type $D_4$.
While there are a number of cases to deal with, this approach also provides a fairly straightforward way to determine the little Weyl group.
We will deal with this in subsequent work.
\end{rk}

\bibliography{biblio}

\end{document}